\documentclass{article}

\usepackage{graphicx}
\usepackage{amssymb, amsmath, amsthm, mathtools, mathrsfs}
\usepackage{amsfonts, bm}
\usepackage[margin=2.5cm]{geometry}
\usepackage[hyphens]{url}
\usepackage[hidelinks]{hyperref}
\usepackage{xcolor}
\usepackage{enumitem}
\usepackage{longtable}
\usepackage{array}
\usepackage{authblk}

\theoremstyle{plain}
\newtheorem{Th}{Theorem}
\newtheorem{Lemma}[Th]{Lemma}
\newtheorem{Cor}[Th]{Corollary}
\newtheorem{Prop}[Th]{Proposition}

\newcommand{\doi}[1]{\url{https://doi.org/#1}}

\newcommand{\R}{\mathcal R}
\newcommand{\Sg}{\mathcal S}

\newcommand{\nR}[1]{n_{\mathcal R_{#1}}}
\newcommand{\DR}[1]{\Delta_{\mathcal R_{#1}}}

\title{The maximum spectral radius of connected graphs \\ with prescribed order and size\thanks{This research was supported by the Science Fund of the Republic of Serbia, project no.~6767, Lazy walk counts and spectral radius of threshold graphs (LZWK).}}
\author{Dragan Stevanovi\'c}
\affil{}
\date{}
\begin{document}
\maketitle
\begin{abstract}
In 1986, Brualdi and Solheid asked which connected graph of prescribed order and size has maximum adjacency spectral radius.
Here we give a complete answer for every feasible pair of parameters.
We prove that no graph outside the dense and star threshold families is extremal, including at ties.
We combine an exact Bellman recursion with dense-core domination, 
fractional complete-split bounds, and a virtual-star inequality.
Three remaining comparison inequalities are established analytically and by exact polynomial and Sturm certificates.
The exact certificates are implemented independently in Python/SymPy and SageMath.
\end{abstract}

\section{Introduction}
For a simple graph $G$, we denote by $\rho(G)$ the largest eigenvalue of its adjacency matrix. 
In 1986, Brualdi and Solheid~\cite{brusol1986} asked for 
the largest spectral radius among connected graphs of fixed order $n$ and size $m$.
They proved that every connected graph attaining the maximum is a threshold graph. 
A graph is \emph{threshold} if its vertices can be ordered so that 
each vertex is adjacent either to all or to none of its predecessors~\cite{chvaham1977}.
Equivalently, when the vertices are ordered by nonincreasing degree, 
its adjacency matrix is \emph{stepwise}: 
$a_{ij}=1$ with $i<j$ implies $a_{k\ell}=1$ whenever $k<\ell$, $k\le i$ and $\ell\le j$
(see also~\cite{mahpel1995}).
Rowlinson~\cite{rowlinson1988} solved the corresponding problem for all graphs,
without the requirement for connectedness.

Writing 
\[
m=n-1+e,
\]
the problem for connected graphs had previously been solved 
for $0\le e\le3$~\cite{brusol1986} and for $e=\binom r2$, $r\ge4$~\cite{bell1991}.
For $e=\binom r2-1$, Olesky, Roy and van den Driessche~\cite{oleroydri2002} established 
the extremal graph whenever $2n\le m<\binom n2-1$.
More recently, Damnjanovi\'c~\cite{damnjanovic2025filomat} resolved the problem 
for all $4\le e\le130$ and every feasible order $n$, 
as well as for $e>130$ whenever $n\ge e+2+13\sqrt e$.

Let
\[
e=\binom{k}{2}+t,\qquad 0\le t<k.
\]
For $e\ge1$, let $\mathcal R_e$ be the graph 
obtained from $K_k$ by adding a vertex adjacent to precisely $t$ vertices of the clique when $t>0$; 
for $t=0$ we simply take $\mathcal R_e=K_k$.
Hence $|V(\mathcal R_e)|=k$ when $t=0$ and $k+1$ when $t>0$.
Let $\mathcal S_e=K_{1,e}$.  
Rowlinson's theorem~\cite{rowlinson1988} states that, among all graphs with $e$ edges, 
the spectral radius is maximized by $\mathcal R_e$ together with any number of isolated vertices.

We call $(n,e)$ \emph{feasible} if a connected graph with $n$ vertices and $n-1+e$ edges exists, 
that is, $0\le e\le\binom{n-1}{2}$. 
For $e\ge1$ this is equivalent to $n\ge |V(\mathcal R_e)|+1$.
Let $G\vee H$ denote the join of two graphs $G,H$.
The two conjectured connected candidates for the maximum spectral radius~\cite{bell1991,oleroydri2002} are
\[
 D_{n,e}=K_1\vee[\mathcal R_e\cup(n-1-|V(\mathcal R_e)|)K_1]
\]
and, when $n\ge e+2$,
\[
 V_{n,e}=K_1\vee[\mathcal S_e\cup(n-e-2)K_1].
\]
For $e=0$ the extremal graph is the star $K_{1,n-1}$, and we set $D_{n,0}=K_{1,n-1}$.

For $e\ge4$, Damnjanovi\'c~\cite{damnjanovic2025filomat}
gave an explicit criterion in terms of the graph order for comparing the two candidates.
He introduced the cubic polynomial
\begin{equation}\label{eq:damnjanovic_cubic}
\begin{aligned}
\Psi_e(x)&={}
x^3(k-1)(k-2)(k^2-3k+4t)\\
&-x^2[k^5-6k^4+k^3(4t+15)-k^2(20t+18)+k(8t^2+24t+8)-(4t^2+12t)]\\
&-x(k^2-k+2t)[k^3+k^2(t-4)-k(3t-5)+(4t^2-2t-2)]\\
&+t(k-t-1)(k^2-3k+2t)(k^2-k+2t).
\end{aligned}
\end{equation}
Let $\psi_e$ be the largest real root of $\Psi_e$, and set
\begin{equation}\label{eq:omega_e}
\omega_e=e+2+
\frac{\psi_e(\psi_e+1)(\psi_e^2-\psi_e-2e)}
{\psi_e^2-e}.
\end{equation}
Damnjanovi\'c proved that $\omega_e>e+2$ and, 
when both graphs exist, that $D_{n,e}$ has larger spectral radius for $n<\omega_e$,
while $V_{n,e}$ has larger spectral radius for $n>\omega_e$;
their spectral radii coincide precisely when $n=\omega_e$.
Ties between the two candidates do occur---%
for example, the cone identity of Lemma~\ref{lem:core_cone} gives the following exact equalities:
\[
 \rho(D_{80,6})=\rho(V_{80,6})=9,\qquad
 \rho(D_{60,10})=\rho(V_{60,10})=8,\qquad
 \rho(D_{68,21})=\rho(V_{68,21})=9.
\]

Our main result is the complete resolution of the Brualdi--Solheid question.

\begin{Th}[All extremal connected graphs]\label{thm:main}
Let $n\ge2$ and $n-1\le m\le\binom n2$, and write $e=m-n+1$.
If $e\le3$, then $D_{n,e}$ is the unique connected graph with maximum spectral radius.  
If $e\ge4$, then the unique extremal graph is $D_{n,e}$ when $n<\omega_e$, and $V_{n,e}$ when $n>\omega_e$.
When $n=\omega_e$, these are the only two extremal graphs, up to isomorphism.  
In particular, for every feasible order and size, no graph other than $D_{n,e}$ or $V_{n,e}$ can be extremal.
\end{Th}

Damnjanovi\'c's Propositions~4.1--4.2~\cite{damnjanovic2025filomat}
compare the two candidates whenever both exist.
We prove that no other graph can be extremal and that $D_{n,e}$ is the unique maximizer for $n\le e+1$.

\paragraph{Method and relation to earlier work.}
The basic objects of our proof are classical.
For a graph $H$, 
the quantity $\lambda\mathbf1^\top(\lambda I-A(H))^{-1}\mathbf1=\sum_{j\ge0}N_j(H)\lambda^{-j}$, 
where $N_j(H)$ counts the walks of length $j$, is 
Cvetkovi\'c's walk generating function $\sum_{j\ge0}N_j(H)\,t^j$ evaluated at $t=1/\lambda$~\cite{cvetkovic1970,cds1980}. 
The poles of $\mathbf1^\top(\lambda I-A(H))^{-1}\mathbf1$ are the main eigenvalues of $H$~\cite{rowlinson2007}. 
The function $\mathbf1^\top(\lambda I-A(H))^{-1}\mathbf1$ is also known 
as the \emph{coronal}~\cite{mclmcn2011} and 
as the \emph{main function} in the spectral theory of $H$-joins~\cite{sarmurar2021}.
The cone identity (Lemma~\ref{lem:core_cone}) is the case $G_1=K_1$ of 
the classical formula for the characteristic polynomial of a join $G_1\vee G_2$. 
The complement resolvent (Lemma~\ref{lem:core_complement}) is a resolvent form of 
Cvetkovi\'c's relation between the walk generating functions of a graph and its complement~\cite{cvetkovic1970,cds1980}.
What is new is the way these functions are used.
We fix the spectral parameter $\lambda$ and treat the walk function $W_H(\lambda)$ as 
the quantity to be maximized over all threshold cores with a prescribed number of edges and a prescribed order cap. 
Applied inside the core, the universal-vertex case of the join formula becomes 
an exact Bellman recursion over strict partitions (Theorem~\ref{thm:core_bellman}). 
The walk functions of the dense core and of the fractional complete-split core then 
bound all intermediate states through the capped bound~\eqref{eq:core_cs_capped}, 
and the remaining comparisons reduce to sign conditions on explicit rational functions, which we certify exactly.
We are not aware of earlier work that uses walk generating functions in this way 
for an extremal spectral problem.

The dependencies of the proof are summarized below. Here an ``exact certificate''
means a symbolic identity, sign certificate or root-exclusion test using exact arithmetic.

\begin{center}
\small
\renewcommand{\arraystretch}{1.16}
\begin{tabular}{@{}>{\raggedright\arraybackslash}p{0.27\linewidth}>{\raggedright\arraybackslash}p{0.38\linewidth}>{\raggedright\arraybackslash}p{0.27\linewidth}@{}}
\hline
Step & Mathematical role & Verification status \\
\hline
Core walk functions and Bellman recursion (Section~\ref{subsec:core_bellman}) & Translate spectral maximality into walk-function comparison at one parameter & Analytic; quotient identities cross-checked symbolically \\
Dense-core domination (Section~\ref{subsec:core_dense_certificates}) & Eliminate all cores with first part $d\le k_g$ & Analytic compression argument \\
Fractional complete-split bound (Section~\ref{subsec:core_capped_envelope}) & Bound all remaining Bellman states by dense or split terms & Analytic, conditional on (J), (K1), (K2) \\
(K2), (J), (K1) (Section~\ref{subsec:core_uniform_checks}) & Certify the three transfer comparisons & (K2) mostly analytic; exact certificates for finite (K2), (J), (K1) \\
Strict exclusion (Theorem~\ref{thm:strict_pointwise}) & Rule out ties with any intermediate core & Analytic consequence of the same induction and strict certificate signs \\
Virtual star (Theorem~\ref{thm:vstar}) & Handle $n\le e+1$ where $V_{n,e}$ does not exist & Analytic reduction; four positive cones and 1766 exact Sturm cases \\
Final choice (Section~\ref{subsec:core_main_complete}) & Identify which of $D_{n,e}$ and $V_{n,e}$ wins & Damnjanovi\'c's Propositions~4.1--4.2, plus the preceding exclusions \\
\hline
\end{tabular}
\end{center}

The exact scripts, expected success markers, and the analytic/computational boundaries 
appear in Appendix~\ref{app:section6_runners} and the supplied proof-to-code map. 
The accompanying verification archive can be found in \cite{zenodo}.
Python/SymPy and SageMath construct the same certificate families in separate exact-arithmetic frameworks; 
the algebraic certificate design is shared.
All computer-assisted results are established by exact algebraic certificates; 
no general conclusion is inferred from testing finitely many instances.

\section{Core resolvents and a Bellman recursion}
\label{subsec:core_bellman}

Every connected threshold graph has a
universal vertex, and we may write it uniquely up to isomorphism as
\[
 G=K_1\vee[H\cup(n-1-|V(H)|)K_1],
\]
where $H$ is a threshold graph with $e=m-n+1$ edges and no isolated
vertices.  We call $H$ the \emph{core}.

For any graph $H$ and $\lambda>\rho(H)$ (with $\lambda>0$ for an edgeless graph), set
\begin{equation}\label{eq:core_walk_function}
 W_H(\lambda)=\sum_{j\ge1}\frac{\mathbf1^\top A(H)^j\mathbf1}{\lambda^j}
 =\lambda\mathbf1^\top(\lambda I-A(H))^{-1}\mathbf1-|V(H)|.
\end{equation}
Adding isolated vertices to $H$ does not change $W_H$.
Thus $W_H(\lambda)+|V(H)|$ is the walk generating function $\sum_{j\ge0}N_j(H)\,t^j$ of~$H$ 
evaluated at $t=1/\lambda$~\cite{cvetkovic1970,cds1980}.
Its coefficients are nonnegative. 
Hence $W_H$ is strictly decreasing on $(\rho(H),\infty)$ whenever $H$ has an edge, 
and $W_H=0$ for an edgeless graph.

For an integer \(g\ge1\), write \(\displaystyle g=\binom{k}{2}+t\), where \(0\le t<k\), 
and let \(\displaystyle h=k-t=\binom{k+1}{2}-g\). 
If $t>0$, the vertices of $\mathcal R_g$ admit an equitable partition into the $t$ clique vertices 
adjacent to the additional vertex, the remaining $h$ clique vertices, and the additional vertex itself. 
The corresponding quotient matrix of the adjacency matrix of $\mathcal R_g$ is
\[
Q_g^{\mathrm{eq}}=
\begin{pmatrix}
t-1 & h   & 1\\
t   & h-1 & 0\\
t   & 0   & 0
\end{pmatrix}.
\]
Consequently, the walk function of $\mathcal R_g$ is
\[
W_{\mathcal R_g}(\lambda)
=\lambda
\begin{pmatrix}t&h&1\end{pmatrix}
(\lambda I_3-Q_g^{\mathrm{eq}})^{-1}
\begin{pmatrix}1\\1\\1\end{pmatrix}
-(k+1).
\]
The row vector $(t,h,1)$ records the sizes of the three classes.
For $t=0$, we have $\mathcal R_g=K_k$ and $\displaystyle W_{\mathcal R_g}(\lambda)=\frac{k(k-1)}{\lambda-k+1}$.

In either case, simplification gives
\begin{equation}\label{eq:WR}
W_{\mathcal R_g}(\lambda)
=\frac{\nR{g}(\lambda)}{\DR{g}(\lambda)},
\end{equation}
where
\begin{align*}
\nR{g}(\lambda)
&=\lambda^2[k(k+1)-2h]+\lambda[2k(k+1)-h(2k+3-h)]+(k+1)[k-h(k+1-h)],\\
\DR{g}(\lambda)
&=(\lambda+1)^2(\lambda-k)+h(\lambda+k+1-h).
\end{align*}
The polynomial $\DR{g}(\lambda)$ is positive for
$\lambda>\rho(\mathcal R_g)$, since for $t>0$
it equals $\det(\lambda I_3-Q_g^{\mathrm{eq}})$, while for $t=0$
it factors as $\lambda(\lambda+1)(\lambda-k+1)$.
The same formulas apply to $\mathcal R_{g'}$ with parameters $k',h'$.

The cubic $\Psi_e$ from~\eqref{eq:damnjanovic_cubic}
can also be expressed in terms of the core walk functions:
\begin{equation}\label{eq:omega_polynomial}
\Psi_e(x)=4\left[
e(2x+e+1)\DR{e}(x)
-(x^2-e)\nR{e}(x)
\right].
\end{equation}
Indeed, since
\[
W_{\mathcal S_e}(x)
=\frac{e(2x+e+1)}{x^2-e},
\]
we have
\[
W_{\mathcal S_e}(x)-W_{\mathcal R_e}(x)
=\frac{\Psi_e(x)}
{4(x^2-e)\DR{e}(x)}.
\]
The terms of degree four cancel because
the leading coefficient of $\nR{e}(x)$ is $2e$,
and expanding the remaining terms gives precisely
the cubic in~\eqref{eq:damnjanovic_cubic}.
Thus $\Psi_e(x)=0$ characterizes the equality
of the two core walk functions for
$x>\rho(\mathcal R_e)$.
By Damnjanovi\'c~\cite[Lemma~4.4]{damnjanovic2025filomat},
the largest real root satisfies $\psi_e>k+1$.
His Propositions~4.1--4.2 establish the comparison
properties of~\eqref{eq:omega_e}.

We start with an elementary spectral identity valid for any graph with a universal vertex.

\begin{Lemma}[Cone identity]\label{lem:core_cone}
If $G=K_1\vee(H\cup pK_1)$ has $n$ vertices, 
then $\lambda=\rho(G)$ is the unique solution, on the branch $\lambda>\rho(H)$, of
\begin{equation}\label{eq:core_cone_equation}
 \lambda^2-(n-1)=W_H(\lambda).
\end{equation}
\end{Lemma}

\begin{proof}
Let $v_H=|V(H)|$, so that $n=1+v_H+p$.
Ordering the vertices as the universal vertex, the vertices of $H$,
and the $p$ isolated vertices, we have
\[
A(G)=
\begin{pmatrix}
0 & \mathbf1^\top & \mathbf1^\top\\
\mathbf1 & A(H) & 0\\
\mathbf1 & 0 & 0
\end{pmatrix},
\]
where the all-ones vectors and zero blocks have the appropriate dimensions.
Write $\lambda=\rho(G)$ and normalize its Perron eigenvector
as $(1,y^\top,z^\top)^\top$, with $y\in\mathbb R^{v_H}$
and $z\in\mathbb R^p$.
The eigenvalue equations give
\[
\lambda=\mathbf1^\top y+\mathbf1^\top z,
\qquad
(\lambda I-A(H))y=\mathbf1,
\qquad
\lambda z=\mathbf1.
\]
Since $G$ is connected and $A(H)$ is a proper principal submatrix
of $A(G)$, we have $\lambda>\rho(H)$.
Therefore
\[
y=(\lambda I-A(H))^{-1}\mathbf1,
\qquad
z=\frac1\lambda\mathbf1,
\]
and hence
\[
\lambda
=\mathbf1^\top(\lambda I-A(H))^{-1}\mathbf1
+\frac p\lambda.
\]
Multiplying by $\lambda$ and using $n-1=v_H+p$ gives
\[
\lambda^2-(n-1)
=\lambda\mathbf1^\top
(\lambda I-A(H))^{-1}\mathbf1-v_H
=W_H(\lambda).
\]
Finally, $W_H(\lambda)$ is strictly decreasing for
$\lambda>\rho(H)$ when $H$ has an edge, and is identically
zero if $H$ is edgeless; in either case it is nonincreasing,
as follows from its walk-series expansion
\eqref{eq:core_walk_function},
whereas $\lambda^2-(n-1)$ is strictly increasing.
Thus the solution on this branch is unique.
\end{proof}

Lemma~\ref{lem:core_cone} is the case $G_1=K_1$, $G_2=H\cup pK_1$ of 
the classical formula for the characteristic polynomial of a join~\cite{cds1980}. 
In terms of the coronal $\Gamma_G(x)=\mathbf1^\top(xI-A(G))^{-1}\mathbf1$~\cite{mclmcn2011}, 
this formula can be written as
\[
P_{G_1\vee G_2}(x)=P_{G_1}(x)\,P_{G_2}(x)\,\bigl(1-\Gamma_{G_1}(x)\,\Gamma_{G_2}(x)\bigr);
\]
see~\cite{sarmurar2021} for its $H$-join generalization. 
Since $\Gamma_{K_1}(x)=1/x$ and $\Gamma_{H\cup pK_1}(x)=\Gamma_H(x)+p/x$, 
the last factor vanishes exactly when~\eqref{eq:core_cone_equation} holds.
We include the short proof because it identifies the Perron branch, on which all later comparisons take place.

The cone identity~\eqref{eq:core_cone_equation} allows us to compare cores of different orders at a
common value of~$\lambda$, with the prescribed graph order $n$ appearing only on the left-hand side.

The parameter $\omega_e$ in~\eqref{eq:omega_e} has a simple interpretation. 
For $e\ge4$, we have $\psi_e>\rho(\mathcal R_e)$ and, by~\eqref{eq:omega_polynomial},
\[
W_{\mathcal R_e}(\psi_e)=W_{\mathcal S_e}(\psi_e).
\]
Solving the cone equation~\eqref{eq:core_cone_equation} for $n$ at $\lambda=\psi_e$ gives
\[
n
=1+\psi_e^2-W_{\mathcal S_e}(\psi_e)
=1+\psi_e^2
-\frac{e(2\psi_e+e+1)}{\psi_e^2-e}
=\omega_e.
\]
Thus $\omega_e$ is the (not necessarily integral) order at which 
the two candidate cores give the same spectral parameter $\psi_e$. 
In particular, if $\omega_e$ is an integer, 
then $D_{\omega_e,e}$ and $V_{\omega_e,e}$ have the same spectral radius $\psi_e$.

\begin{Lemma}[Encoding cores by strict partitions]
\label{lem:core_strict_partitions}
Let $H$ be the core of a connected threshold graph, and suppose that $H$ has $e\ge1$ edges.
The possible cores $H$, up to isomorphism, are in one-to-one correspondence with the strict partitions
\[
e=r_1+\cdots+r_s,\qquad r_1>\cdots>r_s\ge1.
\]
The core corresponding to such a partition has $r_1+1$ vertices. 
If these vertices are labelled $1,\ldots,r_1+1$ in nonincreasing order of their degrees, 
the entries above the diagonal of its stepwise adjacency matrix $(a_{ij})$ satisfy
\[
a_{ij}=1
\quad\Longleftrightarrow\quad
1\le i\le s,\quad i+1\le j\le i+r_i.
\]
Consequently, the restriction $|V(H)|\le n-1$ is equivalent to $r_1\le n-2$.
\end{Lemma}

\begin{proof}
Let $H$ be a core with $q$ vertices, labelled in nonincreasing order of their degrees. 
Since $H$ is also a threshold graph, its adjacency matrix is stepwise. 
Thus, in each row of the upper triangle of this matrix,
all nonzero entries appear consecutively, 
starting immediately to the right of the diagonal.

Let $s$ be the number of rows containing nonzero entries above the diagonal, and 
let $b_i$ denote the column of the last nonzero entry in row $i$, for $1\le i\le s$. 
By the stepwise property,
\[
b_1\ge b_2\ge\cdots\ge b_s.
\]
The number of nonzero entries in row $i$ is $r_i=b_i-i$. 
Therefore,
\[
r_i-r_{i+1}
=b_i-b_{i+1}+1\ge1,
\]
which shows that $r_1>\cdots>r_s\ge1$.
Since each edge of $H$ is represented by exactly one entry above the diagonal, 
we have $e=r_1+\cdots+r_s$.
Moreover, $H$ has no isolated vertices,
so its last vertex $q$ has a neighbour, say $i<q$.
The stepwise property applied to $a_{iq}=1$
forces $a_{1q}=1$, because $1\le i$.
Applying the same property to this entry
shows that $a_{1j}=1$ for every $2\le j\le q$.
Thus the first vertex is universal,
$b_1=q$, and $q=r_1+1$.

Conversely, let $e=r_1+\cdots+r_s$ be a strict partition.
Construct a graph on $r_1+1$ vertices by placing ones in columns $i+1,\ldots,i+r_i$ of row~$i$
of the upper triangle of its adjacency matrix for $1\le i\le s$, 
and zeros in all remaining positions above the diagonal. 
The entries below the diagonal are determined by symmetry.
Since the partition is strict,
\[
r_i\le r_1-i+1
\quad\text{and}\quad
i+r_i\ge i+1+r_{i+1},
\]
so all prescribed entries belong to the matrix and the resulting adjacency matrix is stepwise.
It therefore represents a threshold graph.
Its first vertex is adjacent to every other vertex,
so the graph has no isolated vertices and can serve as the core of a connected threshold graph.
It has $e$ edges by construction.

The two constructions are inverse to each other, up to isomorphism, which proves the correspondence.
\end{proof}

We may therefore optimize over strict partitions instead of threshold cores. Their first part determines
the core order, while deleting it gives the smaller partition used in Theorem~\ref{thm:core_bellman}.

In this notation $\mathcal S_e$ corresponds to $(e)$, 
while $\mathcal R_e$ corresponds to $(k-1,k-2,\ldots,1)$ for $t=0$, 
and to
\[
 (k,k-1,\ldots,k-t+1,k-t-1,\ldots,1)
\]
for $t>0$.  Empty ranges are understood to be omitted.

\begin{Prop}[A common spectral comparison value]\label{prop:core_comparison}
Fix feasible $(n,e)$ with $e\ge1$, and let $L>\rho(\mathcal R_e)$.
For a threshold core $H$ with $e$ edges and $|V(H)|\le n-1$,
\[
 \rho\!\left(K_1\vee(H\cup(n-1-|V(H)|)K_1)\right)\le L
 \quad\Longleftrightarrow\quad
 W_H(L)\le L^2-(n-1).
\]
Consequently the maximum spectral radius of a connected graph of order $n$ and size $m=n-1+e$ is 
the unique $\lambda>\rho(\mathcal R_e)$ satisfying
\begin{equation}\label{eq:core_global_root}
 \lambda^2-(n-1)=\Phi(e,n-1;\lambda),
\end{equation}
where $\Phi(e,N;\lambda)$ is the maximum of $W_H(\lambda)$ over threshold cores with $e$ edges and at most $N$ vertices.
\end{Prop}

\begin{proof}
By Rowlinson's theorem on the maximum spectral radius of graphs with a prescribed number of edges~\cite{rowlinson1988}, 
we have $\rho(H)\le\rho(\mathcal R_e)$ for every $e$-edge core~$H$.
Consequently, the assumption $L>\rho(\mathcal R_e)$ ensures that $L>\rho(H)$, 
so the walk function $W_H(L)$ and the cone identity of Lemma~\ref{lem:core_cone} are both used on their resolvent branch.

Let $G_H=K_1\vee[H\cup(n-1-|V(H)|)K_1]$ and let $\lambda_H=\rho(G_H)$.  
The cone identity~\eqref{eq:core_cone_equation} gives
\[
 \lambda_H^2-(n-1)=W_H(\lambda_H).
\]
Since $x^2-(n-1)$ is strictly increasing for $x>0$, whereas $W_H(x)$ is decreasing for $x>\rho(H)$, 
the difference $x^2-(n-1)-W_H(x)$ is strictly increasing on this interval.
It follows that $\lambda_H\le L$ if and only if $W_H(L)\le L^2-(n-1)$, proving the first assertion.

There are only finitely many cores with $e$ edges
and at most $n-1$ vertices, up to isomorphism.
Their maximum walk function $\Phi(e,n-1;x)$ is continuous
and decreasing for $x>\rho(\mathcal R_e)$,
because it is the maximum of finitely many such functions.
Let $L_D=\rho(D_{n,e})>\rho(\mathcal R_e)$.
Since the dense core is admissible, the cone identity gives
\[
\Phi(e,n-1;L_D)\ge W_{\mathcal R_e}(L_D)=L_D^2-(n-1).
\]
As $x\to\infty$, every one of the finitely many
core walk functions tends to zero,
so $\Phi(e,n-1;x)\to0$.
By continuity, $x^2-(n-1)$ therefore meets
$\Phi(e,n-1;x)$ at least once on this interval.
The intersection is unique,
because the former function is strictly increasing
and the latter is decreasing. 
By the equivalence already proved,
this intersection is the largest spectral radius among the corresponding connected threshold graphs. 
Finally, Brualdi and Solheid~\cite{brusol1986} showed that 
every connected graph attaining the maximum spectral radius for fixed order and size is a threshold graph. 
Hence the same intersection gives the maximum over all connected graphs, as claimed.
\end{proof}

By Proposition~\ref{prop:core_comparison}, it suffices to determine the maximum core walk function
$\Phi(e,n-1;\lambda)$, from which the extremal spectral radius follows by~\eqref{eq:core_global_root}. We
obtain this maximum by an exact recursion. An \emph{order cap} is an upper bound on the number of
vertices of a core.

\begin{Th}[Exact Bellman recursion]\label{thm:core_bellman}
Let $e\ge1$, $N\ge1$, and $\lambda>\rho(\mathcal R_e)$.
Set $\Phi(0,N;\lambda)=0$ for $N\ge0$, and regard a maximum over an empty collection as $-\infty$.
For $d\ge1$, define the \emph{Bellman transfer}
$\mathscr F_\lambda(d,w)$ by
\begin{equation}\label{eq:core_bellman_transfer}
 \mathscr F_\lambda(d,w)
 :=\lambda\left[
 \frac{(1+\gamma)(\lambda+1)}{\lambda-\gamma}-1
 \right]-(d+1),
 \qquad
 \gamma=\frac{d+w}{\lambda}.
\end{equation}

Let $H$ be a threshold core with $e$ edges, and let $d$ be the first part of its strict partition. 
Deleting the corresponding universal vertex and removing any isolated vertices leaves 
a possibly empty core $H_2$ with $e-d$ edges and at most $d$ vertices.
Then
\begin{equation}\label{eq:core_bellman_identity}
 W_H(\lambda)
 =\mathscr F_\lambda(d,W_{H_2}(\lambda)).
\end{equation}
Consequently,
\begin{equation}\label{eq:core_bellman_recurrence}
 \Phi(e,N;\lambda)=
 \max_{\substack{1\le d\le e\\d+1\le N}}
 \mathscr F_\lambda\bigl(d,\Phi(e-d,d;\lambda)\bigr),
\end{equation}
where infeasible inner states are omitted.
The recursion is exact; for rational $\lambda$ its states and transfers use rational arithmetic.
\end{Th}

\begin{proof}
Let $H$ be a core with $e$ edges, and let $d$ be the first part of its strict partition.  
By Lemma~\ref{lem:core_strict_partitions}, 
$H$ has $d+1$ vertices and its first vertex is adjacent to all the remaining $d$ vertices.  
Removing this universal vertex deletes $d$ edges.  
After removing any isolated vertices that remain, 
we obtain a possibly empty core $H_2$ with $e-d$ edges and $b=|V(H_2)|\le d$ vertices.  
Conversely, every such $H_2$ gives a core with $e$ edges and $d+1$ vertices by the construction
\[
 H=K_1\vee[H_2\cup(d-b)K_1].
\]
Hence the possible first parts satisfy $1\le d\le e$ and $d+1\le N$, and the smaller core has order cap~$d$.

We now express $W_H(\lambda)$ in terms of $W_{H_2}(\lambda)$.
Let $a$ be the coordinate corresponding to the universal vertex of the vector $x=(\lambda I-A(H))^{-1}\mathbf1$.
Writing $y$ for its coordinates on $H_2$, the resolvent equations give
\[
 (\lambda I-A(H_2))y=(1+a)\mathbf1
 \quad\Rightarrow\quad
 y=(1+a)(\lambda I-A(H_2))^{-1}\mathbf1.
\]
Each of the other $d-b$ vertices is adjacent only to the universal vertex, 
so its coordinate in $x$ equals $(1+a)/\lambda$.  
By the definition of the walk function, with $W_{H_2}=0$ when $H_2$ is empty, 
the sum $S$ of all coordinates of $x$ other than $a$ is
\[
 S=(1+a)\mathbf1^\top
       (\lambda I-A(H_2))^{-1}\mathbf1
       +\frac{(d-b)(1+a)}{\lambda}
  =\frac{1+a}{\lambda}(d+W_{H_2}(\lambda))
   =(1+a)\gamma,
\]
where $\displaystyle\gamma=\frac{d+W_{H_2}(\lambda)}{\lambda}$.
The equation at the universal vertex now becomes 
\[
\lambda a=1+S=1+(1+a)\gamma, 
\]
whence
$\displaystyle a=\frac{1+\gamma}{\lambda-\gamma}$.
The denominator is positive.
Indeed, by Rowlinson's theorem,
$\lambda>\rho(\mathcal R_e)\ge\rho(H)$,
so $\lambda I-A(H)$ is positive definite.
Writing $B=A(H_2\cup(d-b)K_1)$
for the adjacency matrix on the other $d$ vertices,
the Schur complement of $\lambda I-B$ in
$\lambda I-A(H)$ equals
\[
\lambda-\mathbf1^\top(\lambda I-B)^{-1}\mathbf1
=\lambda-\gamma.
\]
A Schur complement of a positive-definite matrix
is positive, hence $\lambda-\gamma>0$.  
Using $\mathbf1^\top x=a+S=(\lambda+1)a-1$, we obtain
\[
 W_H(\lambda)
 =\lambda(a+S)-(d+1)
 =\lambda\left[
    \frac{(1+\gamma)(\lambda+1)}{\lambda-\gamma}-1
   \right]-(d+1)
 =\mathscr F_\lambda(d,W_{H_2}(\lambda)).
\]
Thus the Bellman transfer expresses the walk function of the larger core $H$ 
in terms of that of the smaller core $H_2$, when a universal vertex is added.

For fixed $d$ and $\lambda$, the transfer function is strictly increasing in its second argument, 
since
\[
 \frac{\partial\mathscr F_\lambda(d,w)}{\partial w}
 =\frac{\lambda^2(\lambda+1)^2}
        {(\lambda^2-d-w)^2}>0.
\]
For an admissible inner core, the denominator is nonzero
because $\lambda^2-d-w=\lambda(\lambda-\gamma)>0$,
as shown above.
Thus the transfer is increasing on the relevant branch.
For a fixed first part $d$, we maximize $W_H(\lambda)$ by maximizing $W_{H_2}(\lambda)$ 
among cores with $e-d$ edges and at most $d$ vertices.  
This inner maximum is $\Phi(e-d,d;\lambda)$.  
Maximizing next over all admissible first parts $d$ proves~\eqref{eq:core_bellman_recurrence}.
Since the converse construction produces a core from every admissible smaller core, 
no possibilities have been lost---the recursion is exact.
\end{proof}

The transfer~\eqref{eq:core_bellman_transfer} is the universal-vertex case of the join formula 
recalled after Lemma~\ref{lem:core_cone}, applied inside the core.
The recursion built from it is a Bellman recursion in the sense of dynamic programming:
after the first part $d$ is chosen, the best continuation is determined by the optimal value of a smaller state.
It finds the maximum walk function and a maximizing core at a \emph{prescribed} spectral parameter $\lambda$, 
but it does not by itself solve the spectral radius equation~\eqref{eq:core_global_root}.
The following sections use this exact recurrence to obtain inequalities valid for all admissible cores, 
rather than to enumerate them individually.

\section{Dense cores}
\label{subsec:core_dense_certificates}

To compare the walk functions of dense cores, it is convenient to consider the edges missing from a complete graph. 
The following resolvent identity applies to this comparison even when the graph formed by the missing edges is not threshold.

\begin{Lemma}[Complement resolvent]\label{lem:core_complement}
Suppose $H=K_{\nu}-F$ and $\lambda>\rho(H)$, where $F$ is a graph on the same $\nu$ vertices.  
With $\beta=\lambda+1$, set
\[
 \alpha_F=\mathbf1^\top(\beta I+A(F))^{-1}\mathbf1.
\]
Then $0<\alpha_F<1$ and
\begin{equation}\label{eq:core_complement_resolvent}
 W_H(\lambda)=\frac{\lambda\alpha_F}{1-\alpha_F}-\nu.
\end{equation}
When $\beta>\rho(F)$,
\[
 \alpha_F=\sum_{j\ge0}\frac{(-1)^j
 \mathbf1^\top A(F)^j\mathbf1}{\beta^{j+1}}.
\]
\end{Lemma}

\begin{proof}
Since $A(H)=J-I-A(F)$, where $J=\mathbf1\mathbf1^\top$, we have
\[
\lambda I-A(H)
=(\beta I+A(F))-\mathbf1\mathbf1^\top.
\]
Let $B=\beta I+A(F)$, so that $\lambda I-A(H)=B-\mathbf1\mathbf1^\top$.
As $\lambda>\rho(H)$, the matrix $\lambda I-A(H)$ is positive definite.
Hence $B=\lambda I-A(H)+\mathbf1\mathbf1^\top$ is positive definite as well, and 
consequently
\[
\alpha_F=\mathbf1^\top B^{-1}\mathbf1>0.
\]
To see that $\alpha_F<1$, let $v=B^{-1}\mathbf1$. 
By the positive definiteness of $\lambda I-A(H)$,
\[
0<v^\top(\lambda I-A(H))v
=v^\top(B-\mathbf1\mathbf1^\top)v
=\alpha_F-\alpha_F^2.
\]
Hence $0<\alpha_F<1$.

Recall now the Sherman--Morrison identity
\[
(B-uu^\top)^{-1}
=B^{-1}
+\frac{B^{-1}uu^\top B^{-1}}
{1-u^\top B^{-1}u},
\]
which expresses the inverse of a rank-one perturbation of $B$ in terms of $B^{-1}$.
Applying it with $u=\mathbf1$ to $\lambda I-A(H)=B-\mathbf1\mathbf1^\top$ gives
\[
\mathbf1^\top(\lambda I-A(H))^{-1}\mathbf1
=\alpha_F+\frac{\alpha_F^2}{1-\alpha_F}
=\frac{\alpha_F}{1-\alpha_F}.
\]
Substituting this into the definition of $W_H(\lambda)$ proves \eqref{eq:core_complement_resolvent}.

Finally, when $\beta>\rho(F)$, the geometric series expansion for the matrix inverse gives
\[
(\beta I+A(F))^{-1}
=\frac1\beta
\left(I+\frac{A(F)}{\beta}\right)^{-1}
=\sum_{j\ge0}
\frac{(-1)^jA(F)^j}{\beta^{j+1}}.
\]
Multiplying by $\mathbf1^\top$ and $\mathbf1$ gives the asserted alternating walk series for $\alpha_F$.
\end{proof}

Since $H=K_\nu-F$ is the complement of $F$ on $\nu$ vertices and $\alpha_F=-\Gamma_F(-\lambda-1)$, 
identity~\eqref{eq:core_complement_resolvent} is equivalent to
\[
\Gamma_{\overline F}(\lambda)=-\frac{\Gamma_F(-\lambda-1)}{1+\Gamma_F(-\lambda-1)}.
\]
This is Cvetkovi\'c's relation between the walk generating functions, 
equivalently the characteristic polynomials, of a graph and its complement~\cite{cvetkovic1970,cds1980}. 
What we add is the bound $0<\alpha_F<1$ on the Perron branch, 
and the use of $\alpha_F$, a walk function at the negative argument $-\beta$, 
in Lemmas~\ref{lem:core_negative_star_threshold} and~\ref{lem:core_negative_compression}.
Since the expression in \(\eqref{eq:core_complement_resolvent}\) is strictly increasing in \(\alpha_F\in(0,1)\), 
comparing the walk functions of graphs obtained by deleting a fixed number of edges from \(K_{\nu}\) 
reduces to comparing the corresponding values of \(\alpha_F\).

For a graph $F$ with $f$ edges and $\beta>\rho(F)$, define the \emph{negative-argument walk function}
\begin{equation}\label{eq:negative_core_walk}
 W_F(-\beta)=\beta\mathbf1^\top(\beta I+A(F))^{-1}\mathbf1-|V(F)|.
\end{equation}
Isolated vertices do not affect this quantity.  
The complement formula becomes
\begin{equation}\label{eq:negative_core_duality}
 \alpha_F=\frac{\nu+W_F(-\beta)}{\beta},\qquad
 W_{K_{\nu}-F}(\lambda)=\frac{\lambda\alpha_F}{1-\alpha_F}-\nu,
 \qquad \beta=\lambda+1,
\end{equation}
so at fixed $\nu$ maximizing the dense-core walk function is equivalent 
to maximizing the negative-argument walk function of the missing edges.
To exploit this relation, 
we first prove that the star maximizes the negative-argument walk function 
among threshold graphs with a prescribed number of edges.  
We then show that every nonthreshold graph can be improved by moving one of its edges, 
so that the same conclusion holds for all simple graphs. 
Together, these two lemmas provide the comparison of missing edges needed for the dense-core domination theorem.

\begin{Lemma}[Negative transfer and the star bound for threshold graphs]
\label{lem:core_negative_star_threshold}
Let $f\ge1$, $\beta\ge f$, with $\beta>1$ when $f=1$.
For every threshold graph $F$ with $f$ edges,
\begin{equation}\label{eq:negative_star_bound}
 W_F(-\beta)\le X_f(\beta):=-\frac{f(2\beta-f-1)}{\beta^2-f}
 =W_{K_{1,f}}(-\beta),
\end{equation}
with equality only for a star together with isolated vertices.
\end{Lemma}

\begin{proof}
Adding isolated vertices does not change the walk function,
so we may assume that $F$ has no isolated vertices.
We first note that the negative-argument walk functions appearing below are well defined. 
Indeed, for $f\ge2$, every graph with $f$ edges has spectral radius smaller than $f$. 
To see this, its maximum degree is at most $f$.
If the maximum degree is $f$, the graph consists of $K_{1,f}$ and possibly isolated vertices, and 
its spectral radius is $\sqrt f<f$. 
Otherwise its maximum degree, and hence its spectral radius, is at most $f-1$.
For $f=1$, the assumption $\beta>1$ gives the same conclusion. 
Thus $\beta>\rho(F)$ in every case.

Let $d$ be the first part of the strict partition representing $F$. 
Removing the corresponding universal vertex and any isolated vertices that remain 
gives a smaller core $F_2$ with $g=f-d$ edges. 
The proof of Theorem~\ref{thm:core_bellman} established the transfer identity
\[
W_F(\lambda)
=\mathscr F_\lambda(d,W_{F_2}(\lambda)),
\]
where $\mathscr F_\lambda$ is defined in~\eqref{eq:core_bellman_transfer}. 
Since this is an identity of rational functions,
it also holds for $\lambda=-\beta$.
Indeed, $\beta>\rho(F)$, as shown above,
and the same spectral-radius bound holds for the smaller core $F_2$;
hence both resolvent matrices at $-\beta$ are nonsingular. 
Writing $w=W_{F_2}(-\beta)$ and $s=d+w$, substitution into \eqref{eq:core_bellman_transfer} gives
\begin{equation}\label{eq:core_negative_transfer}
\mathscr F^-_\beta(d,w)
:=\mathscr F_{-\beta}(d,w)
=\frac{\beta[\beta+(\beta-2)(d+w)]}
        {\beta^2-d-w}-(d+1)
=G_\beta(d+w)-(d+1),
\end{equation}
where
\[
G_\beta(s)
=-\beta(\beta-2)
+\frac{\beta^2(\beta-1)^2}{\beta^2-s}.
\]
The advantage of the last expression is that the dependence on the remaining core 
is contained in the single variable $s=d+w$. In particular,
\[
\frac{\partial\mathscr F^-_\beta(d,w)}
     {\partial w}
=\frac{\beta^2(\beta-1)^2}
       {(\beta^2-d-w)^2}>0
\]
whenever $\beta^2-d-w>0$. 
Thus the negative transfer is strictly increasing in the walk function of the remaining core, on the indicated branch.

We now prove the assertion by induction on $f$.
For $f=1$, the only strict partition is $(1)$, so $F=K_2=K_{1,1}$. 
For $f=2$, the only strict partition is $(2)$, so $F=K_{1,2}$. 
Consequently, the assertion holds with equality in both cases.
More generally, solving the resolvent equations at the centre and leaves of a star gives
\[
W_{K_{1,f}}(-\beta)
=-\frac{f(2\beta-f-1)}{\beta^2-f}
=X_f(\beta).
\]

Suppose next that $f\ge3$ and that the assertion holds for every smaller number of edges. 
If $F$ is a star, equality again holds.
Otherwise, Lemma~\ref{lem:core_strict_partitions} shows that 
the first part of its strict partition satisfies $2\le d\le f-1$. 
Hence the remaining core has $g=f-d\ge1$ edges.

By the induction hypothesis,
\[
w=W_{F_2}(-\beta)\le X_g(\beta).
\]
Since $\beta\ge f>g$, the expression for $X_g$ shows that $X_g(\beta)<0$. 
In particular, $\beta^2-d-w\ge\beta^2-d-X_g>0$.
The monotonicity of the negative transfer therefore gives
\[
W_F(-\beta)
=\mathscr F^-_\beta(d,w)
\le\mathscr F^-_\beta(d,X_g(\beta)).
\]
It remains to show that the expression on the right is strictly smaller than $X_f(\beta)$.

To simplify the difference, let
\[
u=\beta^2-f,\qquad v=g-X_g(\beta).
\]
Both quantities are positive: $u>0$ follows from $\beta\ge f\ge3$, while $v>0$ follows from $X_g(\beta)<0$. 
Moreover, since $f=d+g$,
\[
u+v=\beta^2-d-X_g(\beta).
\]
Thus $u$ and $u+v$ are precisely the two denominators that arise in the comparison.
Substituting the formulas for $X_f$, $X_g$ and $\mathscr F^-_\beta$, and collecting terms, gives
\begin{equation}\label{eq:negative_star_gap}
X_f(\beta)-\mathscr F^-_\beta(d,X_g(\beta))
=\frac{gQ^{\mathrm{neg}}_{f,g}(\beta)}
       {(\beta^2-g)u(u+v)},
\end{equation}
where
\begin{align*}
Q^{\mathrm{neg}}_{f,g}(\beta)
=2(d-2)\beta^4+2(g+2)\beta^3
-\bigl[d(f+2g)+g+1\bigr]\beta^2
+2fg\beta+fg(d-g-1).
\end{align*}
All factors in the denominator of \eqref{eq:negative_star_gap} are positive.
We shall therefore prove that $Q^{\mathrm{neg}}_{f,g}(\beta)>0$.
Since $d\ge2$ and $\beta\ge f$, we have
\[
2(d-2)\beta^4\ge2(d-2)f^2\beta^2,
\qquad
2(g+2)\beta^3\ge2(g+2)f\beta^2.
\]
Consequently,
\[
Q^{\mathrm{neg}}_{f,g}(\beta)
\ge C\beta^2+2fg\beta+fg(d-g-1),
\]
where
\[
C=2(d-2)f^2+2(g+2)f
  -d(f+2g)-g-1.
\]
For $d=2$, substituting $f=g+2$ gives
\[
C=2g^2+g+3>0.
\]
For $d\ge3$, we have $2(d-2)\ge2$, and substituting $f=d+g$ yields
\[
C\ge d^2+3dg+4g^2+4d+3g-1>0.
\]
The remaining two terms satisfy
\[
2fg\beta+fg(d-g-1)
\ge fg(2f+d-g-1)
=fg(3d+g-1)>0.
\]
Hence $Q^{\mathrm{neg}}_{f,g}(\beta)>0$, and \eqref{eq:negative_star_gap} implies
\[
W_F(-\beta)
\le\mathscr F^-_\beta[d,X_g(\beta)]
<X_f(\beta)
\]
whenever $F$ is not a star.

We have thus established the required bound, with strict inequality for every nonstar core.
Since isolated vertices do not affect the walk function, 
equality for an arbitrary threshold graph is possible only for $K_{1,f}$ together with isolated vertices.
\end{proof}

\begin{Lemma}[Negative-resolvent compression]
\label{lem:core_negative_compression}
Under the hypotheses of Lemma~\ref{lem:core_negative_star_threshold},
\eqref{eq:negative_star_bound} holds for \emph{every} simple graph $F$ with $f$ edges.  
Equality holds only for $K_{1,f}$ together with isolated vertices.
\end{Lemma}

\begin{proof}
The case $f=1$ is immediate, so suppose that $f\ge2$.
A graph with $f$ edges has at most $2f$ non-isolated vertices. 
We may therefore add isolated vertices so that all graphs under consideration have the same number $\nu=2f$ of vertices.

Recall that
\[
W_F(-\beta)
=\beta\alpha_F-\nu,\qquad
\alpha_F=\mathbf1^\top(\beta I+A(F))^{-1}\mathbf1.
\]
Adding an isolated vertex increases $\alpha_F$ by $1/\beta$, but also increases $\nu$ by one. 
Thus it does not change $W_F(-\beta)$. 
Since $\nu$ is now fixed,
maximizing $W_F(-\beta)$ is equivalent to maximizing $\alpha_F$.
There are finitely many graphs on this fixed vertex set
with exactly $f$ edges, so a maximizer exists.
Choose such a maximizing graph $F$ and suppose,
for a contradiction, that it is not threshold.
We shall construct another graph with the same order and size
but a strictly larger value of $\alpha_F$.

First, observe that $\beta I+A(F)$ is positive definite. 
Indeed, if the maximum degree of $F$ is at most $f-1$, then $\rho(F)\le f-1<\beta$. 
Otherwise, a vertex of degree $f$ is incident with every edge of $F$,
so $F$ is a star together with isolated vertices.
In this case $\rho(F)=\sqrt f<f\le\beta$.
Therefore $\beta>\rho(F)$ in either case, and all the inverses used below exist.

Let
\begin{equation}\label{eq:w-definition}
w=(\beta I+A(F))^{-1}\mathbf1.
\end{equation}
We first show that every coordinate of $w$ is positive. 
To this end, consider the map
\[
T(v)=\frac{\mathbf1-A(F)v}{\beta}
\]
on the cube $C=[0,1/\beta]^{\nu}$.
Because $F$ is not threshold,
it cannot have a vertex of degree $f$:
such a vertex would be incident with every edge,
so $F$ would be a star together with isolated vertices,
which is threshold.
Thus $\Delta(F)\le f-1<\beta$.
For every $v\in C$ and every vertex $a$, we have
\begin{equation}\label{eq:Tv-lower-bound}
0<\frac{1-\Delta(F)/\beta}{\beta}
\le T(v)_a\le\frac1\beta.
\end{equation}
Thus $T$ maps $C$ into itself. Moreover, for any $v,v'\in C$,
\[
\|T(v)-T(v')\|_\infty
=\frac1\beta\|A(F)(v-v')\|_\infty
\le\frac{\Delta(F)}{\beta}
       \|v-v'\|_\infty.
\]
Because $\Delta(F)/\beta<1$, the map $T$ is a contraction in the maximum norm.
By the contraction mapping theorem, it has a unique fixed point in $C$. 
The fixed-point equation is equivalent to
\[
(\beta I+A(F))w=\mathbf1,
\]
so this fixed point is precisely the vector $w$ defined in~\eqref{eq:w-definition}. 
The lower bound~\eqref{eq:Tv-lower-bound} for $T(v)_a$ also shows that $w_a$ is positive for every vertex $a$.

Since the maximizing graph $F$ was assumed not to be threshold,
the vicinal-order characterization
(see, for example,~\cite{mahpel1995})
provides two vertices $i,j$ such that neither of the sets
\[
N_F(i)\setminus\{j\},
\qquad
N_F(j)\setminus\{i\}
\]
is contained in the other. 
We may assume that $w_i\le w_j$, since otherwise we can interchange the roles of $i$ and $j$.
As their neighborhoods are incomparable, there exists a vertex
\[
r\in N_F(j)\setminus(N_F(i)\cup\{i\}).
\]
Thus $jr$ is an edge of $F$, whereas $ir$ is not. 
Let $F'=F-jr+ir$. The graph $F'$ is simple and has the same number $f$ of edges.

To compare $\alpha_F$ and $\alpha_{F'}$ without explicitly calculating either inverse,
we use a variational characterization of the inverse of a positive-definite matrix. 
Let $B$ be any real symmetric positive-definite matrix, and let $z_0=B^{-1}\mathbf1$.
Since $Bz_0=\mathbf1$ and $B$ is symmetric, for every $z\in\mathbb R^{\nu}$ we have
\[
(z-z_0)^\top B(z-z_0)
=z^\top Bz-2z^\top Bz_0+z_0^\top Bz_0
=z^\top Bz-2\mathbf1^\top z
  +\mathbf1^\top B^{-1}\mathbf1.
\]
The left-hand side is nonnegative, and it vanishes only when $z=z_0$, by the positive definiteness of $B$. 
Rearranging the preceding identity therefore gives
\[
\mathbf1^\top B^{-1}\mathbf1
=\max_{z\in\mathbb R^{\nu}}
  [2\mathbf1^\top z-z^\top Bz],
\]
where the maximum is attained uniquely at $z=B^{-1}\mathbf1$.

We now apply this identity to $B=\beta I+A(F)$ and $B'=\beta I+A(F')$.
For $B$, the maximizing vector is $w=B^{-1}\mathbf1$, and hence
\[
\alpha_F=2\mathbf1^\top w-w^\top Bw.
\]
For $B'$ we do not know its maximizing vector, but substituting~$w$ into the variational expression we obtain
\[
\alpha_{F'}
\ge 2\mathbf1^\top w-w^\top B'w
=\alpha_F+w^\top(B-B')w
=\alpha_F+2w_r(w_j-w_i)
\ge\alpha_F.
\]
Here the third equality follows from the fact that $F'=F-jr+ir$,
while the last inequality follows from $w_r>0$ and our assumption $w_i\le w_j$.

If $w_i<w_j$, the last inequality is strict, because $w_r>0$. 
If $w_i=w_j$, the value of the quadratic form at $w$ remains unchanged,
but $w$ is no longer the maximizing vector for $B'$. 
Indeed, the edge replacement changes the $i$-th and $j$-th coordinates of $Bw$ by $w_r$ and $-w_r$, respectively.
Since $w_r>0$, we have
\[
B'w\ne Bw=\mathbf1.
\]
Thus $w\ne(B')^{-1}\mathbf1$, and the uniqueness of the maximizer in the variational identity 
implies $\alpha_{F'}>\alpha_F$ also in this case.

We have therefore constructed a graph $F'$ with the same number of vertices and edges as $F$, but with $\alpha_{F'}>\alpha_F$.
This contradicts the assumption that $F$ maximizes $\alpha_F$. 
Consequently, every maximizing graph must be threshold.

Lemma~\ref{lem:core_negative_star_threshold} now applies to every maximizing graph and
shows that its negative-argument walk function is at most $X_f(\beta)$, 
with equality only for $K_{1,f}$ together with isolated vertices.
Since a star attains this value, the asserted inequality and its equality condition follow.
\end{proof}

The preceding lemmas give a pointwise walk-function comparison for all cores with at most $k+1$
non-isolated vertices, without restricting the number of edges missing from $K_{k+1}$. This strengthens
Rowlinson's spectral-radius comparison~\cite{rowlinson1988}.

\begin{Th}[Dense-core domination for any number of missing edges]
\label{thm:core_dense_all}
Let $\displaystyle e=\binom{k}{2}+t$, $0\le t<k$, and 
let $H$ be any simple graph with $e$ edges and at most $k+1$ non-isolated vertices.  
For every $\lambda>\rho(\mathcal R_e)$,
\begin{equation}\label{eq:core_dense_all}
 W_H(\lambda)\le W_{\mathcal R_e}(\lambda),
\end{equation}
with equality only when $H\cong\mathcal R_e$ after removal of isolated vertices.
\end{Th}

\begin{proof}
Add isolated vertices to $H$, if necessary, so that it has $\nu=k+1$ vertices. 
Adding isolated vertices does not change its walk function. 
We may then write $H=K_{k+1}-F$, where the graph of missing edges $F$ has
\[
 f=\binom{k+1}{2}-e=k-t\le k
\]
edges. Since $\rho(\mathcal R_e)\ge k-1$, the hypothesis on $\lambda$ implies $\beta=\lambda+1>k\ge f$.
Lemma~\ref{lem:core_negative_compression} consequently yields $\alpha_F\le\alpha_{K_{1,f}}$, 
with equality only when $F$ is a star together with isolated vertices.

By Lemma~\ref{lem:core_complement}, the function $W_H(\lambda)=\lambda\alpha_F/(1-\alpha_F)-(k+1)$ is 
strictly increasing as a function of $\alpha_F\in(0,1)$.
Thus it is at most the walk function obtained by deleting a star of $f$ edges from $K_{k+1}$. 
The resulting graph is precisely $\mathcal R_e$,
with one isolated vertex added when $t=0$:
the $k$ vertices other than the star centre form a clique,
while that centre has exactly $k-f=t$ neighbours in the clique. 

Moreover, equality can hold only if $\alpha_F=\alpha_{K_{1,f}}$. 
By the equality condition in Lemma~\ref{lem:core_negative_compression},
this requires $F$ to be a star together with isolated vertices. 
Since $F$ has $k+1$ vertices and $f=k-t$ edges, we obtain
\[
F\cong K_{1,f}\cup tK_1.
\]
Consequently, $H$ is isomorphic to $\mathcal R_e$ when $t>0$, or to $\mathcal R_e\cup K_1$ when $t=0$.
This proves the equality assertion.
\end{proof}

In particular, Theorem~\ref{thm:core_dense_all} does not require $H$ or its complement to be threshold.

\begin{Cor}[A dense connected order]
\label{cor:core_dense_min_order}
Let $\displaystyle e=\binom{k}{2}+t$, where $k\ge2$ and $0\le t<k$, and let $n=k+2$ and $m=n-1+e$.
Then $D_{n,e}$ is the unique connected graph of order $n$ and size $m$ with maximum spectral radius, up to isomorphism.

Moreover, if $H$ is the core of a connected threshold graph with $e$ edges, and 
the first part of its corresponding strict partition satisfies $d\le k$, then
\[
W_H(\lambda)\le W_{\mathcal R_e}(\lambda)
\qquad\text{for every }\lambda>\rho(\mathcal R_e),
\]
with equality only when $H\cong\mathcal R_e$.
\end{Cor}

\begin{proof}
For $n=k+2$, every admissible core has at most $n-1=k+1$ vertices. 
By Theorem~\ref{thm:core_dense_all}, 
its walk function is therefore bounded above by $W_{\mathcal R_e}(\lambda)$ for every $\lambda>\rho(\mathcal R_e)$, 
with equality only for the dense core.
Proposition~\ref{prop:core_comparison} then shows that 
$D_{n,e}$ is the unique extremal connected threshold graph, up to isomorphism.
The theorem of Brualdi and Solheid~\cite{brusol1986} extends this conclusion to all connected graphs.

For the second assertion, recall from Lemma~\ref{lem:core_strict_partitions} that
a core whose corresponding strict partition has first part $d$ contains exactly $d+1$ vertices. 
Hence $d\le k$ implies $|V(H)|\le k+1$, and 
the inequality and its equality condition follow again from Theorem~\ref{thm:core_dense_all}.
\end{proof}

Consequently, dense-core domination settles all first parts $d\le k$ in the Bellman recursion; only
$d\ge k+1$ remains.

\section{The virtual-star comparison}
\label{subsec:core_virtual_star}

The pointwise comparison will involve both $W_{\mathcal R_e}$ and $W_{\mathcal S_e}$. Since
$|V(\mathcal S_e)|=e+1$, the graph $V_{n,e}$ exists only for $n\ge e+2$. For smaller orders we must still
control $W_{\mathcal S_e}$, although the star core is not admissible.

To complete the comparison for these smaller orders, we shall prove that, for every feasible $n\le e+1$,
\[
W_{\mathcal R_e}[\rho(D_{n,e})]
>W_{\mathcal S_e}[\rho(D_{n,e})].
\]
Thus the star term cannot determine the upper bound at the spectral radius of the dense candidate. 
We call this the \emph{virtual-star comparison}, 
because it uses the walk function of the star even when the star core has too many vertices to be admissible.

Rather than proving the comparison separately for every order $n$, 
we shall reduce it to an inequality on an interval whose endpoint depends only on~$e$.
For $e\ge4$, let $\mu_e>\sqrt e$ be the unique root of
\begin{equation}\label{eq:core_virtual_mu}
 (\mu_e^2-e)^2=e(2\mu_e+e+1).
\end{equation}
This is the cone equation for the star at the \emph{formal} order $n=e+1$.
The root $\mu_e$ is unique: 
the defining polynomial $\phi(\mu)=(\mu^2-e)^2-e(2\mu+e+1)$ is negative at $\sqrt e$, 
tends to positive infinity, and has $\phi''(\mu)=12\mu^2-4e>0$ for $\mu\ge\sqrt e$.

\begin{Lemma}[Virtual-star comparison]\label{lem:core_virtual_star}
For every $e\ge4$, the parameter $\mu_e$ satisfies
\[
\mu_e<\sqrt{2e}+1.
\]
Suppose, in addition, that
\begin{equation}\label{eq:core_virtual_condition}
W_{\mathcal R_e}(\lambda)>W_{\mathcal S_e}(\lambda)
\quad\text{for }\rho(\mathcal R_e)<\lambda\le\mu_e.
\end{equation}
Then, for every feasible $n\le e+1$,
\[
W_{\mathcal R_e}(\rho(D_{n,e}))>
W_{\mathcal S_e}(\rho(D_{n,e})).
\]
\end{Lemma}

\begin{proof}
We first prove the upper bound for $\mu_e$.
Recall that $\mu_e>\sqrt e$ is the unique root of
\[
\phi_e(\mu)=(\mu^2-e)^2-e(2\mu+e+1)
\]
in the interval $(\sqrt e,\infty)$.
Let $r=\sqrt{2e}$ and $\mu=r+1$. Since $e=r^2/2$, direct substitution gives
\[
\phi_e(r+1)
=[(r+1)^2-e]^2
  -e[2(r+1)+e+1]
=r^3+\frac72r^2+4r+1>0.
\]
On the other hand, $\phi_e(\sqrt e)=-e(2\sqrt e+e+1)<0$.
As $\phi_e$ has exactly one root greater than $\sqrt e$, these inequalities imply
\[
\mu_e<r+1=\sqrt{2e}+1,
\]
proving the first assertion.

We now assume that \eqref{eq:core_virtual_condition} holds.
Our aim is to compare the walk functions of the dense and star cores at $\rho(D_{n,e})$ for every feasible order $n\le e+1$.
The difficulty is that the star core has $e+1$ vertices and therefore cannot occur in a connected graph of order $n\le e+1$.
Nevertheless, we can still consider its cone equation~\eqref{eq:core_cone_equation} formally with $H=\mathcal S_e$.
For any positive integer $n$, let $\lambda_V(n)>\sqrt e$ be the solution of
\begin{equation}\label{eq:core_formal_star_root}
\lambda^2-(n-1)=W_{\mathcal S_e}(\lambda).
\end{equation}
We call $\lambda_V(n)$ the \emph{formal star root}. 
When $n\ge e+2$, it is the spectral radius of $V_{n,e}$,
while for smaller $n$ it is merely the solution of \eqref{eq:core_formal_star_root}.

Let us first explain why this solution exists, is unique, and increases with $n$.
Recall that
\[
W_{\mathcal S_e}(\lambda)=\frac{e(2\lambda+e+1)}{\lambda^2-e}.
\]
For $\lambda>\sqrt e$, this function is positive and strictly decreasing, 
tending to infinity as $\lambda$ approaches $\sqrt e$ from above and to zero as $\lambda$ tends to infinity. 
In contrast, $\lambda^2-(n-1)$ is strictly increasing in $\lambda$. 
Thus the two sides of \eqref{eq:core_formal_star_root} intersect at exactly one value $\lambda_V(n)>\sqrt e$.
Moreover, increasing $n$ decreases the left-hand side of \eqref{eq:core_formal_star_root} 
while leaving the right-hand side unchanged.
Their intersection must therefore move to the right. 
Hence $\lambda_V(n)$ is strictly increasing as a function of $n$.

At the formal order $n=e+1$, equation \eqref{eq:core_formal_star_root} becomes
\[
\lambda^2-e =\frac{e(2\lambda+e+1)}{\lambda^2-e},
\]
or equivalently,
\[
(\lambda^2-e)^2=e(2\lambda+e+1).
\]
This is precisely the defining equation \eqref{eq:core_virtual_mu} for $\mu_e$.
Consequently,
\[
\lambda_V(e+1)=\mu_e,
\]
and the monotonicity just established gives
\begin{equation}\label{eq:core_formal_star_bound}
\lambda_V(n)\le\mu_e
\qquad\text{whenever }n\le e+1.
\end{equation}

Now fix a feasible order $n\le e+1$.
We shall first prove that $\rho(D_{n,e})>\lambda_V(n)$.

There are two cases.
If $\lambda_V(n)\le\rho(\mathcal R_e)$, then $\rho(D_{n,e})>\lambda_V(n)$ follows immediately from
\[
\rho(D_{n,e})>\rho(\mathcal R_e)\ge\lambda_V(n),
\]
since the cone over the dense core has spectral radius strictly larger than that of the core itself.

Suppose therefore that $\lambda_V(n)>\rho(\mathcal R_e)$. 
By~\eqref{eq:core_formal_star_bound}, we have
\[
\rho(\mathcal R_e)<\lambda_V(n)\le\mu_e,
\]
so the assumed inequality \eqref{eq:core_virtual_condition} gives
\begin{equation}\label{eq:dense-core-at-lambda-V}
W_{\mathcal R_e}(\lambda_V(n))>W_{\mathcal S_e}(\lambda_V(n))=\lambda_V(n)^2-(n-1),
\end{equation}
where the equality follows from the definition \eqref{eq:core_formal_star_root} of the formal star root.

On the other hand, the cone identity \eqref{eq:core_cone_equation} for $D_{n,e}$ gives
\[
W_{\mathcal R_e}(\rho(D_{n,e}))=\rho(D_{n,e})^2-(n-1).
\]
Since $W_{\mathcal R_e}(\lambda)$ is strictly decreasing for $\lambda>\rho(\mathcal R_e)$, 
whereas $\lambda^2-(n-1)$ is strictly increasing,
the two sides of this equation can intersect only at a value larger than $\lambda_V(n)$.
Indeed, from \eqref{eq:dense-core-at-lambda-V} 
the dense-core walk function at $\lambda_V(n)$ is still strictly larger than $\lambda_V(n)^2-(n-1)$.
Thus $\rho(D_{n,e})>\lambda_V(n)$ in this case as well.

We have therefore proved, for every feasible $n\le e+1$, that
\[
\rho(D_{n,e})>\lambda_V(n).
\]

It remains to use this inequality to obtain the desired comparison at $\rho(D_{n,e})$.
Since $\lambda_V(n)>\sqrt e$ and the star walk function is strictly decreasing on $(\sqrt e,\infty)$, 
we have
\[
W_{\mathcal S_e}(\rho(D_{n,e}))<W_{\mathcal S_e}(\lambda_V(n)).
\]
By the defining equation for $\lambda_V(n)$, 
the right-hand side equals $\lambda_V(n)^2-(n-1)$, which is strictly smaller than $\rho(D_{n,e})^2-(n-1)$.
Finally, the cone identity \eqref{eq:core_cone_equation} for the dense candidate 
identifies the latter expression with $W_{\mathcal R_e}(\rho(D_{n,e}))$. 
Altogether,
\[
W_{\mathcal S_e}(\rho(D_{n,e}))
<W_{\mathcal S_e}(\lambda_V(n))
=\lambda_V(n)^2-(n-1)
<\rho(D_{n,e})^2-(n-1)
=W_{\mathcal R_e}(\rho(D_{n,e})).
\]
This proves the required virtual-star comparison and completes the proof.
\end{proof}

\subsection{The uniform virtual-star inequality}\label{subsec:core_virtual_star_proof}

Lemma~\ref{lem:core_virtual_star} reduces the virtual-star comparison for all feasible orders $n\le e+1$ 
to the pointwise inequality~\eqref{eq:core_virtual_condition}.
We now prove that this inequality holds for every $e\ge4$, thereby removing the assumption from that lemma.
As before, we write
\[
e=\binom{k}{2}+t,
\qquad 0\le t<k.
\]
The proof uses a uniform polynomial-positivity certificate for $k\ge60$ and 
exact Sturm calculations for the remaining finitely many values of~$e$.
Both methods involve exact algebraic computations only.

\begin{Th}\label{thm:vstar}
For every $e\ge4$, $W_{\R_e}(\lambda)>W_{\Sg_e}(\lambda)$ for $\rho(\R_e)<\lambda\le\mu_e$.
\end{Th}

\begin{proof}
We distinguish two cases according to the size of $k$.
For $k\ge60$, we shall reduce the desired comparison on an entire interval of $\lambda$
to a single rational inequality in $k$ and $t$.
We shall then establish this inequality by a uniform coefficient-positivity certificate.
For $k\le59$, where the same auxiliary inequality does not always hold,
we shall return to the original comparison and verify it by exact Sturm calculations.

\medskip
\noindent
\emph{The case $k\ge60$.}
Recall that $\mathcal R_e$ contains $K_k$ as a subgraph and is a proper subgraph of $K_{k+1}$.
Consequently,
\[
k-1\le\rho(\mathcal R_e)<k.
\]
Moreover, since $0\le t\le k-1$,
\[
(k-1)^2-e
=\frac{(k-1)(k-2)}2-t
\ge\frac{(k-1)(k-4)}2>0.
\]
Thus $k-1>\sqrt e=\rho(\mathcal S_e)$, and both walk functions are well defined at the arguments that we shall use.

Let
\[
U=k+\frac12+\frac tk.
\]
The choice of $U$ is motivated by the identity
\[
\left(k-\frac12+\frac tk\right)^2-2e
=\left(\frac12-\frac tk\right)^2\ge0.
\]
Since $\displaystyle k-\frac12+\frac tk>0$, this identity shows that
\[
U\ge\sqrt{2e}+1>\mu_e,
\]
where the last inequality follows from Lemma~\ref{lem:core_virtual_star}.
We also have $U>k>\rho(\mathcal R_e)$.

Now let
\[
\rho(\mathcal R_e)<\lambda\le\mu_e.
\]
Both walk functions are strictly decreasing above their respective spectral radii.
Since $\lambda\le\mu_e<U$, we therefore have
\[
W_{\mathcal R_e}(\lambda)
\ge W_{\mathcal R_e}(\mu_e)
>W_{\mathcal R_e}(U).
\]
Similarly, since
$\lambda>\rho(\mathcal R_e)\ge k-1>\sqrt e$,
we have
\[
W_{\mathcal S_e}(\lambda)
<W_{\mathcal S_e}(k-1).
\]
Thus, instead of comparing the two walk functions for every $\lambda$ in the required interval,
it is sufficient to prove the single inequality
\begin{equation}\label{eq:vstar_endpoint_comparison}
W_{\mathcal R_e}(U)>W_{\mathcal S_e}(k-1).
\end{equation}
The advantage is that both arguments are now explicit expressions in $k$ and $t$,
so the variable $\lambda$ has disappeared.

Using the rational expression \eqref{eq:WR} for $W_{\mathcal R_e}$ and the formula
\[
W_{\mathcal S_e}(\lambda)
=\frac{e(2\lambda+e+1)}{\lambda^2-e},
\]
an exact simplification gives
\[
W_{\mathcal R_e}(U)-W_{\mathcal S_e}(k-1)
=\frac{\mathcal N_{\mathrm V}(k,t)}{\mathcal D_{\mathrm V}(k,t)},
\]
where
\[
\mathcal D_{\mathrm V}(k,t)
=4[(k-1)^2-e]Q_{\mathrm V}(k,t)
\]
and
\[
Q_{\mathrm V}(k,t)=8k^3t(k-t)+12k^5+24k^4
+28k^3t+9k^3+8k^2t^2
+30k^2t+28kt^2+8t^3.
\]
Every term in $Q_{\mathrm V}(k,t)$ is nonnegative for $k\ge60$ and $0\le t\le k-1$, and the term $12k^5$ is strictly positive.
Together with $(k-1)^2-e>0$, this proves that
\[
\mathcal D_{\mathrm V}(k,t)>0.
\]
Hence \eqref{eq:vstar_endpoint_comparison} is equivalent to
\[
\mathcal N_{\mathrm V}(k,t)>0,
\]
where this numerator is equal to
\begin{align*}
\mathcal N_{\mathrm V}(k,t)={}&4k^9-8k^8t-80k^8+8k^7t^2-76k^7t-33k^7+8k^6t^2-286k^6t+158k^6\\
&+32k^5t^3-364k^5t^2-108k^5t-7k^5-88k^4t^3-640k^4t^2+290k^4t-42k^4\\
&+32k^3t^4-528k^3t^3+64k^3t^2+8k^3t-96k^2t^4-352k^2t^3+192k^2t^2\\
&-208kt^4+160kt^3-32t^5+32t^4.
\end{align*}
Since its coefficients have mixed signs, its positivity is not evident from the above expression.
We establish it by transforming the parameter region into finitely many regions
on which coefficient positivity becomes evident.

First, let
\[
k=60+m,\qquad t=(k-1)s,
\]
where $m\ge0$ and $0\le s\le1$.
These substitutions describe the entire region
$k\ge60$, $0\le t\le k-1$,
including the admissible integer values
that we actually need.

The split into four equal subintervals is a certificate-design choice, not a graph-theoretic threshold.
The single expansion after $t=(k-1)s$ has mixed coefficients; affine changes of origin in $s$ followed by the nonnegative-half-line substitution can have nonnegative coefficients on smaller regions.
Four equal pieces yield positive coefficient cones, as checked exactly by both programs; no optimality or uniqueness of this subdivision is claimed.
The cutoff $k=60$ is likewise a convenient verified boundary for this particular endpoint estimate: the stronger endpoint inequality already fails at $(k,t)=(59,32)$, while the remaining values are handled by exact Sturm exclusion.

We divide $[0,1]$ into the four subintervals
\[
\left[\frac q4,\frac{q+1}{4}\right],
\qquad q=0,1,2,3.
\]
On the $q$-th subinterval, we make the further substitution
\[
s=\frac q4+\frac{y}{4(1+y)},
\qquad y\ge0.
\]
As $y$ increases from zero to infinity, this expression increases from $q/4$ toward $(q+1)/4$.
Thus the four substitutions together cover the required interval for $s$, 
with their common endpoints included by continuity.

Because $\mathcal N_{\mathrm V}(k,t)$ has degree $5$ in $t$,
multiplication by $(1+y)^5$ clears all denominators introduced by the substitution for $s$.
More precisely, for $q=0,1,2,3$, consider the polynomial
\[
P_q(m,y)
=(1+y)^5\,
\mathcal N_{\mathrm V}\left(
60+m,
(59+m)\left(\frac q4+
\frac{y}{4(1+y)}\right)
\right).
\]

Exact symbolic expansion shows that each of the four polynomials $P_q(m,y)$ has $60$ nonzero monomials,
all its coefficients are nonnegative, and its constant coefficient is strictly positive.
Consequently,
\[
P_q(m,y)>0
\qquad\text{for all }m,y\ge0.
\]
Since $(1+y)^5>0$, 
this proves $\mathcal N_{\mathrm V}(k,t)>0$ throughout the interior of each subinterval and at its left endpoint.
The right endpoints of the first three subintervals are left endpoints of the following ones.
The remaining endpoint $s=1$ is covered by the leading coefficient in $y$ of $P_3(m,y)$: 
this coefficient is a polynomial in $m$ with nonnegative coefficients and a strictly positive constant term.
Thus positivity also holds at $s=1$.

We have therefore established
\[
\mathcal N_{\mathrm V}(k,t)>0
\qquad\mbox{for }k\ge60,\ 0\le t\le k-1,
\]
even when $k$ and $t$ are allowed to take real values in the indicated region.
The coefficient checks are exact and are reproduced by the accompanying computer algebra certificates.

It follows that \eqref{eq:vstar_endpoint_comparison} holds.
Combining it with the monotonicity estimates established above, we obtain
\[
W_{\mathcal R_e}(\lambda)
\ge W_{\mathcal R_e}(\mu_e)
>W_{\mathcal R_e}(U)
>W_{\mathcal S_e}(k-1)
>W_{\mathcal S_e}(\lambda)
\]
for every $\rho(\mathcal R_e)<\lambda\le\mu_e$.
This proves the theorem whenever $k\ge60$.

\medskip
\noindent
\emph{The remaining case $k\le59$.}
The preceding coefficient certificate cannot simply be extended to all smaller $k$.
Indeed, the auxiliary numerator can become negative---for example,
\[
\mathcal N_{\mathrm V}(59,32)<0.
\]
This does not mean that the desired virtual-star inequality fails.
It only shows that the stronger endpoint comparison \eqref{eq:vstar_endpoint_comparison} 
is not sufficiently sharp in this range.

We therefore return to the original difference $W_{\mathcal R_e}(\lambda)-W_{\mathcal S_e}(\lambda)$.
Since
\[
e=\binom{k}{2}+t,
\qquad 0\le t<k,
\]
the range $k\le59$, with $e\ge4$, is equivalent to
\[
4\le e\le1769.
\]
There are only finitely many values of $e$ to consider, and 
we shall verify the required inequality for each of them by an exact algebraic test.

For a fixed $e$, write the difference of the two walk functions as
\[
W_{\mathcal R_e}(\lambda)
-W_{\mathcal S_e}(\lambda)
=
\frac{P_e(\lambda)}
{\DR{e}(\lambda)(\lambda^2-e)},
\]
where
\[
P_e(\lambda)
=(\lambda^2-e)\nR{e}(\lambda)
-e(2\lambda+e+1)\DR{e}(\lambda).
\]
Notice that this is the polynomial already encountered in \eqref{eq:omega_polynomial}, 
since
\[
P_e(\lambda)=-\frac14\Psi_e(\lambda).
\]
Thus we only have to establish the positivity of an explicitly known polynomial on the required interval.

For each $4\le e\le1769$,
we first enclose the interval on which $P_e(\lambda)$ must be positive
between two rational numbers:
a lower bound for $\rho(\mathcal R_e)$
and an upper bound for $\mu_e$.
An exact Sturm root count, followed by rational bisection,
gives a positive rational lower bound
\[
\sqrt e<\ell\le\rho(\mathcal R_e).
\]
In particular, $\ell^2>e$, and both walk functions are well defined throughout the interval that we need to examine.
Similarly, rational bisection applied to the defining polynomial
\[
\phi_e(\lambda)
=(\lambda^2-e)^2-e(2\lambda+e+1)
\]
gives an upper bound
\[
\mu^+>\mu_e.
\]

If $\mu^+\le\ell$, the required interval is empty and there is nothing to prove.
Otherwise, the interval
\[
\rho(\mathcal R_e)<\lambda\le\mu_e
\]
is contained in the larger interval $[\ell,\mu^+]$, whose endpoints are rational numbers.

We now use Sturm's theorem, which determines the exact number of distinct real roots of a polynomial in an interval 
from the sign variations of its Sturm sequence at the endpoints.
For every $4\le e\le1769$ with a nonempty interval,
the exact computation establishes that $P_e(\lambda)$ has no real root in $[\ell,\mu^+]$
and that
\[
P_e(\mu^+)>0.
\]
Since a continuous polynomial cannot change sign without passing through zero,
these two facts imply that
\[
P_e(\lambda)>0
\qquad
\text{throughout }[\ell,\mu^+].
\]

Finally, for every $\rho(\mathcal R_e)<\lambda\le\mu_e$,
we have
\[
\DR{e}(\lambda)>0
\]
by \eqref{eq:WR}, while
\[
\lambda^2-e>\ell^2-e>0.
\]
Thus the denominator of $W_{\mathcal R_e}(\lambda)-W_{\mathcal S_e}(\lambda)$ is positive,
and the positivity of $P_e(\lambda)$ gives
\[
W_{\mathcal R_e}(\lambda)
>W_{\mathcal S_e}(\lambda)
\]
throughout the required interval.

All rational bounds, Sturm root counts, and sign evaluations are computed exactly, without floating-point approximations.
They are implemented in \path{python/vstar_certificate.py} and \path{sage/sage_vstar.py}.
Appendix~\ref{app:vstar_verification} explains how the computations can be reproduced.

Together with the uniform certificate for $k\ge60$, 
these exact finite checks prove \eqref{eq:core_virtual_condition} for every $e\ge4$,
completing the proof.
\end{proof}

\section{Fractional complete-split bounds}
\label{subsec:core_capped_envelope}

We now compare the walk functions of threshold cores with $W_{\mathcal R_g}$ and $W_{\mathcal S_g}$.

The Bellman recursion of Theorem~\ref{thm:core_bellman} gives the maximal walk function for a prescribed
edge count and order cap, but does not identify the maximizing cores.
Fix $g\ge4$ and write
\[
g=\binom{k_g}{2}+t_g,
\qquad 0\le t_g<k_g.
\]
Let $d$ be the first part of the strict partition corresponding to a threshold core $H$ with $g$ edges.
By Lemma~\ref{lem:core_strict_partitions}, the core $H$ has exactly $d+1$ vertices.
When $d\le k_g$, Theorem~\ref{thm:core_dense_all} already bounds its walk function by that of the dense core $\mathcal R_g$.
At the other extreme, $d=g$ gives the star core $\mathcal S_g$.
Thus the only cores that remain to be resolved are those whose first part satisfies
\[
k_g+1\le d\le g-1.
\]

We bound these cores by functions that retain both the edge count and the order cap and are preserved by
the Bellman transfer. Indeed, deleting the first universal vertex of a core with first part $d$ leaves a
core with $g-d$ edges and at most $d$ vertices, by Theorem~\ref{thm:core_bellman}.

For a \emph{complete-split graph}, consisting of a clique of size $s$ joined to an independent set of
size $N-s$, the number of edges is
\[
\binom{s}{2}+s(N-s).
\]
Adding a universal vertex preserves this structure and increases $s$ by one.

For prescribed $g$ and $N$, the preceding edge equation need not have an integer solution. We therefore allow a
real-valued clique-size parameter $s$ and call the resulting algebraic extension the \emph{fractional
complete-split core} $\mathcal C_{g,N}$. Here $N$ is an order parameter, used as an upper cap in the
Bellman induction rather than necessarily as the order of a graph. We set
\begin{equation}\label{eq:core_cs_s}
 s=s(g,N):=N-\dfrac12-\sqrt{\left(N-\dfrac12\right)^2-2g},
 \qquad \dbinom{s}{2}+s(N-s)=g.
\end{equation}
For integer $s$, the clique--independent set partition has equitable quotient matrix
\[
Q_{s,N}=
\begin{pmatrix}
s-1 & N-s\\
s & 0
\end{pmatrix}.
\]
Extending $Q_{s,N}$ to real $s$, we define
\[
C_{g,N}(\lambda)
=
\lambda
\begin{pmatrix}s&N-s\end{pmatrix}
(\lambda I-Q_{s,N})^{-1}
\begin{pmatrix}1\\1\end{pmatrix}
-N.
\]
A direct calculation gives the rational expression and its positive pole:
\begin{equation}\label{eq:core_cs_walk}
 C_{g,N}(\lambda):=\frac{2g\lambda+sN(N-s)}
 {\lambda^2-(s-1)\lambda-s(N-s)},\qquad
 r_{g,N}:=\frac{s-1+\sqrt{(s-1)^2+4s(N-s)}}2.
\end{equation}
At $N=g+1$, $s=1$ and $C_{g,g+1}=W_{\mathcal S_g}$.
Adjoining a universal vertex increases the clique-size parameter from $s$ to $s+1$. For $g'=g-d$, this
gives
\begin{equation}\label{eq:core_cs_closure}
 \mathscr F_\lambda(d,C_{g',d}(\lambda))
   =C_{g,d+1}(\lambda)
\end{equation}
wherever both sides are defined.
This is an identity of rational functions obtained from~\eqref{eq:core_cs_walk}; it does not require an
integral clique size.
In particular, writing $\sigma=s(g',d)$, the identity remains valid for $0<\sigma<1$ and $d\ge2$, 
with
\[
g'=\frac{\sigma(2d-\sigma-1)}2,
\qquad g=g'+d,
\]
provided the relevant denominators are nonzero.

For these fractional values,
$C_{g',d}$ and $C_{g,d+1}$ are rational functions rather than walk functions of actual graphs.
We use this algebraically extended family 
to construct auxiliary bounds that are compatible with the Bellman recursion,
not to introduce additional candidates for extremal graphs.
Specifically, our aim is to bound the capped Bellman maximum $\Phi(g,N;\lambda)$
by the larger of $W_{\mathcal R_g}(\lambda)$ and $C_{g,N}(\lambda)$.
The importance of this bound is that, at the largest possible core order $N=g+1$, 
we have
\[
C_{g,g+1}(\lambda)
=W_{\mathcal S_g}(\lambda).
\]
Thus such a bound would show that no intermediate threshold core can have 
a walk function larger than both the dense and star cores.

To propagate these bounds through the Bellman recursion,
we must ensure that the transfer remains on its increasing branch,
also when its argument is a fractional complete-split function.
The following lemma establishes the required monotonicity.

\begin{Lemma}[Increasing branch for fractional bounds]
\label{lem:transfer_branch}
Let $g=d+g'\ge1$, $1\le d\le g$, and
$\lambda>\rho(\mathcal R_g)$.
Suppose that $w=C_{g',d}(\lambda)>0$ is finite and
the closure identity
\[
\mathscr F_\lambda(d,w)=C_{g,d+1}(\lambda)
\]
holds with a finite, positive right-hand side.
Then $w<\lambda^2-d$.
Consequently, the Bellman transfer is strictly increasing
on the interval between $w$ and $W_B(\lambda)$
for every genuine core $B$ with $g'$ edges
and at most $d$ vertices.
\end{Lemma}

\begin{proof}
By the definition of the Bellman transfer in~\eqref{eq:core_bellman_transfer},
\[
C_{g,d+1}(\lambda)+\lambda+d+1
=\frac{\lambda(\lambda+1)(\lambda+d+w)}
       {\lambda^2-d-w}.
\]
The left-hand side and the numerator are positive,
so $\lambda^2-d-w>0$.

For genuine cores, the positivity of the Schur complement $\lambda-\gamma$, 
established in the proof of Theorem~\ref{thm:core_bellman}, 
gives
\[
\lambda^2-d-W_B(\lambda)
=\lambda(\lambda-\gamma)>0,
\qquad
\gamma=\frac{d+W_B(\lambda)}{\lambda}.
\]

Thus both arguments lie below the pole, where the transfer is strictly increasing, 
as established in that theorem.
\end{proof}

The closure identity \eqref{eq:core_cs_closure} makes this approach compatible with the Bellman recursion.
We next establish a pole bound and the unimodality of $C_{g,N}(\lambda)$ as $N$ varies. These properties
will be used in the reduction to (J), (K1), and (K2).

\begin{Th}[The complete-split pole and unimodality bounds]
\label{thm:core_cs_valley}
Let $g\ge4$ and write $g=\binom{k}{2}+t$ with $0\le t<k$.
For every real $k+2\le N\le g+1$,
\begin{equation}\label{eq:core_cs_pole}
 r_{g,N}<\rho(\mathcal R_g).
\end{equation}
For fixed $\lambda>\rho(\mathcal R_g)$, 
if $k+2<g+1$, the function $N\mapsto C_{g,N}(\lambda)$ is 
either strictly monotone or first strictly decreasing and then strictly increasing on $[k+2,g+1]$.
In particular, for every $k+2\le N_0\le N_1\le g+1$,
\begin{equation}\label{eq:core_cs_endpoint}
 \max_{N_0\le N\le N_1}C_{g,N}(\lambda)
 =\max\{C_{g,N_0}(\lambda),C_{g,N_1}(\lambda)\}.
\end{equation}
\end{Th}

\begin{proof}
We first establish the admissible range of the edge parameter $g$ relative to the order cap $N$.
For fixed~$N$, let
\[
E_N(s)=\binom{s}{2}+s(N-s)
      =s\left(N-\frac12\right)-\frac{s^2}{2}.
\]
On one hand $N\le g+1$, so $N-1\le g$.
On the other hand, from $\displaystyle g=\binom{k}{2}+t$ with $0\le t<k$ we have
\(\displaystyle
g\le\binom{k+1}{2}-1<\binom{k+2}{2}\le\binom N2.
\)
Thus
\[
N-1\le g<\binom N2.
\]
Moreover,
\[
E_N(1)=N-1,
\qquad
E_N(N-1)=\binom N2,
\]
while
\[
E_N'(s)=N-\frac12-s
\ge\frac12>0
\qquad (1\le s\le N-1).
\]
Thus the equation $E_N(s)=g$ has exactly one solution in $[1,N-1)$.
This is the smaller root given by \eqref{eq:core_cs_s}.
In particular, throughout the prescribed range,
\[
1\le s<N-1,
\qquad
\tau:=N-s>1.
\]

Implicit differentiation of $E_N(s)=g$ with $g$ fixed gives
\[
s+\left(N-\frac12-s\right)\frac{ds}{dN}=0.
\]
Since $N-s=\tau$, we obtain
\begin{equation}\label{eq:core_cs_s_derivative}
\frac{ds}{dN}
=-\frac{s}{\tau-1/2}<0.
\end{equation}
Thus $s(g,N)$ decreases strictly with $N$.

To prove~\eqref{eq:core_cs_pole}, we first consider the smallest cap $N=k+2$ and then show that $r_{g,N}$
decreases with $N$.

At $N=k+2$, write
\[
s=k-u,\qquad h=k-t.
\]
Here $1\le h\le k$. 
Furthermore, $u>0$. 
Indeed, as shown above, $E_{k+2}(s)$ is strictly increasing on $[1,k+1]$, and
\[
E_{k+2}(k)=\binom{k}{2}+2k>\binom{k}{2}+t=g=E_{k+2}(s).
\]
Hence $s<k$, so $u=k-s>0$.
Substituting $s=k-u$ into $E_{k+2}(s)=g$ and simplifying gives
\begin{equation}\label{eq:core_cs_u_relation}
u^2+3u=2(k+h).
\end{equation}
Notice also that
\[
\tau=k+2-s=u+2.
\]

For a lower bound on $\rho(\mathcal R_g)$, write
\[
A(\mathcal R_g)=
\begin{pmatrix}
J_k-I_k & a\\
a^\top & 0
\end{pmatrix},
\qquad
\mathbf1^\top a=t.
\]
Restricting the Rayleigh quotient to the subspace spanned by the orthonormal vectors
\[
u=\frac1{\sqrt{k}}(\mathbf1_k^\top,0)^\top,
\qquad
v=(\mathbf0_k^\top,1)^\top,
\]
gives the symmetric compression
\[
\begin{pmatrix}
u^\top Au & u^\top Av\\
v^\top Au & v^\top Av
\end{pmatrix}
=
\begin{pmatrix}
k-1 & t/\sqrt{k}\\
t/\sqrt{k} & 0
\end{pmatrix}.
\]
By the Rayleigh principle, its largest eigenvalue is a lower bound for $\rho(\mathcal R_g)$.
Consequently,
\[
\rho(\mathcal R_g)\ge b
:=\frac{k-1+\sqrt{(k-1)^2+4t^2/k}}2.
\]
For $t=0$, the same inequality holds with $b=k-1=\rho(K_k)$.

Recall from \eqref{eq:core_cs_walk} that $r=r_{g,k+2}$ is the positive root of
\[
P_s(Y)=Y^2-(s-1)Y-s\tau.
\]
Since $s>0$ and $\tau>1$, the constant term of $P_s$ is negative, so its two roots have opposite signs.
We shall prove that $P_s(b)>0$.
As $b>0$, this will imply that $b$ lies to the right of the positive root,
and hence that
\[
r_{g,k+2}<b\le\rho(\mathcal R_g).
\]
The definition of $b$ gives
\[
b^2-(k-1)b=\frac{t^2}{k}.
\]
Using this identity together with $s=k-u$, $\tau=u+2$, and \eqref{eq:core_cs_u_relation},
we obtain
\[
P_s(b)
=u(b-k-1)+2h+\frac{t^2}{k}.
\]
Since $b\ge k-1$ and $u>0$,
\[
P_s(b)
\ge-2u+2h+\frac{t^2}{k}
=-2u+k+\frac{h^2}{k},
\]
where the last equality follows from $t=k-h$.

Suppose first that $k\ge4$.
To prove that this lower bound is positive,
let
\[
v=\frac{k+h^2/k}{2}.
\]
Then the bound becomes $2(v-u)$, so it suffices to prove that $v>u$.

By \eqref{eq:core_cs_u_relation}, $u$ is the unique positive solution of
\[
x^2+3x=2(k+h).
\]
The function $x\mapsto x^2+3x$ is strictly increasing for $x>0$.
Thus $v>u$ will follow if
\[
v^2+3v>2(k+h).
\]
Multiplying $v^2+3v-2(k+h)$ with $4k^2$ and expanding gives
\[
4k^2[v^2+3v-2(k+h)]=k^2[(k-4)(k+2)+2(h-2)^2]+h^4+6kh^2.
\]
For $k\ge4$ and $h\ge1$, every term on the right is nonnegative, and the last two terms are strictly positive.
Hence $v^2+3v-2(k+h)>0$, which implies $v>u$ and
\[
P_s(b)\ge2(v-u)>0.
\]
Therefore $r_{g,k+2}<b\le\rho(\mathcal R_g)$ whenever $k\ge4$.

It remains to consider $k=3$. 
Since $g\ge4$ and $0\le t<3$, the only possibilities are $g=4$ and $g=5$.

For $g=4$, we have $N=5$ and $s=1$.
Thus
\[
r_{4,5}=2
<1+\frac{2\sqrt3}3
=b\le\rho(\mathcal R_4).
\]

For $g=5$, the smallest cap is again $N=5$, and
\[
s=s(5,5)=\frac{9-\sqrt{41}}2
\quad\Rightarrow\quad 
1<s<\frac32.
\]
The edge equation \eqref{eq:core_cs_s} gives $s(9-s)=10$.
Substituting this relation into $P_s(5/2)$, we find
\[
P_s(5/2)=\frac{6s-5}{4}>0.
\]
Since $P_s$ has exactly one positive root, it follows that
\[
r_{5,5}<\frac52
<1+\sqrt{\frac73}
=b\le\rho(\mathcal R_5).
\]

We have therefore proved
\[
r_{g,k+2}<\rho(\mathcal R_g)
\]
for every $g\ge4$.

To extend this inequality to larger caps,
we examine how $r_{g,N}$ varies with $s$ along the curve of fixed edge count $g$.
The defining equation $P_s(r)=0$ is
\[
r^2-(s-1)r-s\tau=0.
\]
Since
\[
s\tau=g-\frac{s(s-1)}2,
\]
we can rewrite it as
\[
r^2-g=(s-1)\left(r-\frac s2\right).
\]
Differentiating with respect to $s$, while keeping $g$ fixed, gives
\[
(2r-s+1)\frac{dr}{ds}
=r-s+\frac12,
\]
and hence
\[
\frac{dr}{ds}=\frac{r-s+\frac12}{2r-s+1}.
\]

Both numerator and denominator are positive.
From \eqref{eq:core_cs_walk} the denominator equals
\[
2r-s+1=\sqrt{(s-1)^2+4s\tau}>0.
\]
For the numerator, evaluate $P_s$ at $Y=s-\frac12$:
\[
P_s(s-\tfrac12)=s\left(\frac12-\tau\right)-\frac14<0,
\]
because $\tau>1$.
Since $s-\tfrac12>0$ and $r$ is the unique positive root of $P_s$,
this implies
\[
s-\frac12<r.
\]
Therefore $\displaystyle\frac{dr}{ds}>0$.

Since $r_{g,N}$ increases with $s$ and $s$ decreases with $N$ by~\eqref{eq:core_cs_s_derivative},
\[
r_{g,N}\le r_{g,k+2}
<\rho(\mathcal R_g),
\]
proving \eqref{eq:core_cs_pole} for every permitted order cap.

To prove unimodality, fix $g\ge4$ and $\lambda>\rho(\mathcal R_g)$ and use $s$ as the variable. The edge
equation~\eqref{eq:core_cs_s} gives $\tau=N-s$ as a function of $s$:
\begin{equation}\label{eq:core_cs_tau}
\tau=\tau(s)
=\frac gs-\frac{s-1}{2},
\qquad
\tau'(s)
=-\frac{\tau+s-1/2}{s}.
\end{equation}
Let
\[
\Delta_{\mathrm{pole}}(s)=\lambda^2-(s-1)\lambda-s\tau.
\]
Since $s(g,N)$ decreases with $N$ and $s(g,g+1)=1$, 
the permitted range of $N$ corresponds to
\[
1\le s\le s(g,k+2).
\]
By the pole bound already established,
\[
\lambda>\rho(\mathcal R_g)>r_{g,N}.
\]
Since the other root of $\Delta_{\mathrm{pole}}(s)$, regarded as a polynomial in $\lambda$, is negative,
we have $\Delta_{\mathrm{pole}}(s)>0$ throughout this interval.

We also need
\begin{equation}\label{eq:core_cs_lambda_s}
\lambda>s+1.
\end{equation}

For $k\ge4$, relation \eqref{eq:core_cs_u_relation} gives
\[
u^2+3u=2(k+h)\ge2(k+1)\ge10.
\]
Since $u>0$ and $x^2+3x$ is increasing for $x>0$, we obtain $u\ge2$.
Thus, at the smallest cap $N=k+2$,
\[
s(g,k+2)=k-u\le k-2.
\]
As $s(g,N)$ decreases with $N$, the same bound holds throughout the interval.
Therefore,
\[
\lambda>\rho(\mathcal R_g)
\ge k-1\ge s+1.
\]

For $g=4$, we have $s=1$ and
\[
\lambda>\rho(\mathcal R_4)
\ge1+\frac2{\sqrt3}>s+1.
\]
For $g=5$, the smallest cap is $N=k+2=5$.
Since $s(5,N)$ decreases with $N$, we have
\[
s\le s(5,5)<\frac32,
\]
while
\[
\lambda>\rho(\mathcal R_5)
\ge1+\sqrt{\frac73}>\frac52>s+1.
\]
This proves \eqref{eq:core_cs_lambda_s} in every case.

Now, keeping $g$ and $\lambda$ fixed, 
regard the fractional complete-split walk function $C_{g,N}(\lambda)$ 
as a function of the clique-size parameter $s$:
\[
F(s)=C_{g,s+\tau(s)}(\lambda).
\]
Writing its numerator from \eqref{eq:core_cs_walk} as
\[
A(s)=2g\lambda+s(s+\tau)\tau,
\]
we have 
\[
F(s)=\frac{A(s)}{\Delta_{\mathrm{pole}}(s)}.
\]
Differentiating the edge equation~\eqref{eq:core_cs_s}
and using the relation~\eqref{eq:core_cs_tau}, we obtain
\[
\frac{d}{ds}(s\tau)=\frac12-s,
\qquad
\frac{d}{ds}(s+\tau)=\frac{1/2-\tau}{s}.
\]
These give
\[
\Delta_{\mathrm{pole}}'(s)=s-\lambda-\frac12,
\]
and, by the product rule,
\[
A'(s)
=\frac{d}{ds}[(s\tau)(s+\tau)]
=\left(\frac12-s\right)(s+\tau)
 +\left(\frac12-\tau\right)\tau.
\]
Substituting into the quotient rule gives
\[
F'(s)=\frac{\eta(s)}{2\Delta_{\mathrm{pole}}(s)^2},
\]
where
\[
\eta(s)
=s(2\tau-1)(\lambda+\tau)^2
-2\tau(\tau-1)\lambda(\lambda+1).
\]

Since $\Delta_{\mathrm{pole}}(s)>0$, the sign of $F'(s)$ is the sign of $\eta(s)$.
We determine the sign of $\eta'(s)$ at an interior zero of $\eta$.
Using the expression for $\tau'(s)$ from~\eqref{eq:core_cs_tau},
the total derivative of $\eta$ with respect to~$s$ is
\[
\frac{d\eta}{ds}
=\frac{\partial\eta}{\partial s}
-\frac{\tau+s-\frac12}{s}\frac{\partial\eta}{\partial \tau}.
\]
Rearranging, we obtain
\[
\begin{aligned}
\frac{\eta'(s)}{(\lambda+\tau)^2}
={}& \frac{(2\tau-1)(2\lambda+1)-2s(\tau-1-\lambda)}{\lambda+\tau}
    +\frac{\tau+s-1/2}{\tau(\tau-1)}
    -\frac{(2\tau-1)(2s+2\tau-1)}{2s\tau(\tau-1)(\lambda+\tau)^2}\eta(s),
\end{aligned}
\]
and since $\eta(s)=0$, the last term vanishes, leaving
\[
\frac{\eta'(s)}{(\lambda+\tau)^2}
=\frac{(2\tau-1)(2\lambda+1)
      -2s(\tau-1-\lambda)}
     {\lambda+\tau}
+\frac{\tau+s-1/2}{\tau(\tau-1)}.
\]
The second summand is positive due to $\tau>1$ and $s>0$.
The numerator of the first summand can be rewritten as
\[
(2\tau-1)(2\lambda+1)-2s(\tau-1-\lambda)
=(2\tau-1)(2\lambda+1-s)+s(2\lambda+1)>0,
\]
since $\tau>1$, $s>0$ and $\lambda>s+1$ by \eqref{eq:core_cs_lambda_s}.
Thus both summands are positive, and consequently
\[
\eta'(s)>0
\]
at every interior zero of $\eta$.
Since $F'(s)$ has the sign of $\eta(s)$, 
this shows that, at every interior zero of $F'$,
it changes sign from negative to positive.
hence each interior stationary point is a strict local minimum. 
There can be at most one such point, 
as two minima would require an intervening stationary point with the opposite sign change. 
Thus $F$ is either strictly monotone or strictly decreasing to a unique minimum and then strictly increasing.

Finally, since
\[
\frac{dN}{ds}=\frac{1/2-\tau}{s}<0,
\]
the change of variables from $s$ to $N$ is strictly decreasing.
This reverses the order of the intervals and the derivative signs, so minimum-type unimodality is
preserved. Thus $N\mapsto C_{g,N}(\lambda)$ is either strictly monotone or strictly decreasing to a
minimum and then strictly increasing. Its maximum on any closed subinterval is therefore attained at an
endpoint:
\[
\max_{N_0\le N\le N_1}C_{g,N}(\lambda)
=\max\{C_{g,N_0}(\lambda),
       C_{g,N_1}(\lambda)\}.
\qedhere       
\]
\end{proof}

We apply these bounds to the Bellman maximum $\Phi(g,N;\lambda)$ of
Proposition~\ref{prop:core_comparison}. Here $N$ is the order cap; removing the first universal vertex of
a core with first part $d$ leaves a core with $g-d$ edges and order at most $d$, as
in~\eqref{eq:core_bellman_recurrence}.

Fix $g\ge4$ and $\lambda>\rho(\mathcal R_g)$, and write
\[
g=\binom{k_g}{2}+t_g,\qquad 0\le t_g<k_g.
\]
By Theorem~\ref{thm:core_dense_all}, every $g$-edge core on at most $k_g+1$ vertices
has walk function at most $W_{\mathcal R_g}(\lambda)$.
It therefore remains to control the larger order caps, and our aim is to establish
\begin{equation}\label{eq:core_cs_capped}
\Phi(g,N;\lambda)
\le\max\{W_{\mathcal R_g}(\lambda),C_{g,N}(\lambda)\},
\qquad k_g+2\le N\le g+1,
\end{equation}
for every integer cap $N$.

At $N=g+1$ the cap includes every $g$-edge threshold core.
Moreover, $s(g,g+1)=1$, and hence
$C_{g,g+1}(\lambda)=W_{\mathcal S_g}(\lambda)$ by~\eqref{eq:core_cs_walk}.
Thus~\eqref{eq:core_cs_capped} would give
\begin{equation}\label{eq:core_pointwise_conjecture}
\max_{\substack{H\text{ threshold core}\\|E(H)|=g}}
W_H(\lambda)
=\max\{W_{\mathcal R_g}(\lambda),W_{\mathcal S_g}(\lambda)\}.
\end{equation}
The reverse inequality is immediate, since both distinguished cores belong to the family
on the left.

To prove~\eqref{eq:core_cs_capped} by induction on $g$, we use the endpoint
inequality~\eqref{eq:core_cs_endpoint} of Theorem~\ref{thm:core_cs_valley} and the closure
identity~\eqref{eq:core_cs_closure} for the fractional complete-split functions $C_{g,N}(\lambda)$. Three
further inequalities are needed, denoted by (J), (K1), and (K2).
For $g\ge4$ and $\lambda>\rho(\mathcal R_g)$, these are the following conditions.
\begin{description}[leftmargin=2.4em]
\item[(J)]
If $C_{g,k_g+2}(\lambda)>W_{\mathcal R_g}(\lambda)$, then the values
\[
C_{g,k_g+2}(\lambda),\ C_{g,k_g+3}(\lambda),\ \ldots,\
C_{g,g+1}(\lambda)
\]
are nondecreasing as the \emph{integer} order cap increases.

\item[(K1)]
For $k_g+1\le d\le g-1$ such that $|V(\mathcal R_{g-d})|\le d$,
\[
\mathscr F_\lambda(d,W_{\mathcal R_{g-d}}(\lambda))
\le\max\{W_{\mathcal R_g}(\lambda),C_{g,d+1}(\lambda)\}.
\]

\item[(K2)]
For $\max\{k_g+1,(g+1)/2\}\le d\le g-1$,
\[
\mathscr F_\lambda
(d,W_{\mathcal S_{g-d}}(\lambda))
\le C_{g,d+1}(\lambda).
\]
\end{description}

Condition (J), with~\eqref{eq:core_cs_endpoint}, controls the intermediate caps. Condition (K1) compares
the transfer of the smaller dense core, provided that it fits the inner cap $d$, whereas (K2) handles the
smaller star core when $d\ge(g+1)/2$, equivalently $d\ge g'+1$ for $g'=g-d$. Both transfer conditions
concern $d\ge k_g+1$, the range not covered by dense-core domination. The fractional term transfers
exactly by~\eqref{eq:core_cs_closure}, and the stronger form of (K2) will yield the first-part
estimate~\eqref{eq:core_firstpart_strong}.

The next theorem gives the conditional reduction; Section~\ref{subsec:core_uniform_checks} verifies its
hypotheses.

\begin{Th}[Three-family conditional reduction]
\label{thm:core_cs_conditional}
Let $G\ge4$, and suppose that (J), (K1), and (K2) hold
for every integer $4\le g\le G$ and every $\lambda>\rho(\mathcal R_g)$.
Then the capped bound~\eqref{eq:core_cs_capped} holds for every
$4\le g\le G$, every integer $k_g+2\le N\le g+1$,
and every $\lambda>\rho(\mathcal R_g)$.
Moreover, for each such $g$ and $\lambda$, if $H$ is a $g$-edge threshold core
whose first part satisfies $k_g+1\le d\le g-1$,
and if $g'=g-d$, then the same induction gives the stronger \emph{first-part estimate}
\begin{equation}\label{eq:core_firstpart_strong}
W_H(\lambda)\le
\max\{\mathscr F_\lambda(d,W_{\mathcal R_{g'}}(\lambda)),\,
C_{g,d+1}(\lambda)\}.
\end{equation}
In particular, the pointwise identity~\eqref{eq:core_pointwise_conjecture} holds for every $1\le g\le G$.
\end{Th}

The first-part estimate~\eqref{eq:core_firstpart_strong} will also be used in the proof of Theorem~\ref{thm:strict_pointwise} 
to establish strict domination of intermediate threshold cores.

\begin{proof}
We begin by deriving a preliminary estimate for intermediate order caps 
from condition (J) and Theorem~\ref{thm:core_cs_valley}.
Fix $4\le g\le G$, $\lambda>\rho(\mathcal R_g)$,
and an integer cap $N$ with $k_g+2\le N\le g+1$.
For every real $N'$ in the interval $[k_g+2,N]$,
the endpoint bound~\eqref{eq:core_cs_endpoint}, applied to this interval, gives
\[
C_{g,N'}(\lambda)
\le\max\{C_{g,k_g+2}(\lambda),C_{g,N}(\lambda)\}.
\]
If $C_{g,k_g+2}(\lambda)\le W_{\mathcal R_g}(\lambda)$,
the right-hand side is at most $\max\{W_{\mathcal R_g}(\lambda),C_{g,N}(\lambda)\}$.
Otherwise (J) applies, and because both $k_g+2$ and $N$ are integer caps,
we have $C_{g,k_g+2}(\lambda)\le C_{g,N}(\lambda)$.
Thus, in either case, we obtain the following intermediate-cap estimate:
\begin{equation}\label{eq:core_cs_prefix}
\max_{k_g+2\le N'\le N}C_{g,N'}(\lambda)
\le\max\{W_{\mathcal R_g}(\lambda),C_{g,N}(\lambda)\}.
\end{equation}
The maximum on the left is over \emph{real} $N'$.
Although (J) concerns only integer caps, the endpoint bound~\eqref{eq:core_cs_endpoint}
yields~\eqref{eq:core_cs_prefix} for every real intermediate cap $N'$, including $N'=d+1$ in the Bellman
recursion.

We now prove~\eqref{eq:core_cs_capped} by induction on the number of edges.
For $g=1,2$, the only threshold core is the star, which is also the dense core.
For $g=3$, the two possibilities are $K_3=\mathcal R_3$ and
$K_{1,3}=\mathcal S_3$, and a direct calculation gives
\[
W_{\mathcal R_3}(\lambda)-W_{\mathcal S_3}(\lambda)
=\frac{6}{(\lambda-2)(\lambda^2-3)}>0
\qquad(\lambda>2).
\]
These direct comparisons supply the base cases for smaller edge counts.
The capped bound itself is required only for $g\ge4$.

Let $4\le g\le G$ and assume that the capped bound~\eqref{eq:core_cs_capped} holds 
for every integer $4\le g'<g$, with all admissible order caps and spectral parameters.
Fix an arbitrary $\lambda>\rho(\mathcal R_g)$.
For every $1\le g'<g$, Rowlinson's theorem~\cite{rowlinson1988},
as used in Proposition~\ref{prop:core_comparison}, yields
\[
\rho(\mathcal R_{g'})\le\rho(\mathcal R_g)<\lambda.
\]
Thus both the induction hypothesis for $g'\ge4$ and the direct comparisons for $g'\le3$ can be applied
at the spectral parameter $\lambda$.

Consider a threshold core $H$ with $g$ edges and at most $N$ vertices,
where $N$ is any integer in $[k_g+2,g+1]$.
Let $d$ be the first part of its strict partition.
By Lemma~\ref{lem:core_strict_partitions},
\[
|V(H)|=d+1\le N.
\]
If $d\le k_g$, Theorem~\ref{thm:core_dense_all}
gives $W_H(\lambda)\le W_{\mathcal R_g}(\lambda)$.
If $d=g$, then $H=\mathcal S_g$ and $d+1=g+1\le N$ forces $N=g+1$;
in this case $W_H(\lambda)=C_{g,N}(\lambda)$ by~\eqref{eq:core_cs_walk}.
It therefore remains to consider first parts in the range
\[
k_g+1\le d\le g-1.
\]

Write $g'=g-d\ge1$.
Removing the first universal vertex of $H$, and deleting any isolated vertices,
leaves a threshold core $H_2$ with $g'$ edges and at most $d$ vertices.
In particular, the inner state $(g',d)$ is feasible.
By the transfer identity~\eqref{eq:core_bellman_identity} in Theorem~\ref{thm:core_bellman}, 
we have
\[
W_H(\lambda)
=\mathscr F_\lambda(d,W_{H_2}(\lambda)).
\]
Moreover, by the definition of $\Phi$,
\[
W_{H_2}(\lambda)\le\Phi(g',d;\lambda).
\]
Since the dense core $\mathcal R_{g'}$ has the minimum possible order among graphs with $g'$ edges, 
feasibility also implies
\[
|V(\mathcal R_{g'})|\le d.
\]
Thus the order hypothesis in (K1) holds.

Before applying the Bellman transfer to the inner-core bounds,
we verify that all relevant arguments lie below its pole $w=\lambda^2-d$, 
where the transfer is strictly increasing.
For genuine inner cores, including $\mathcal R_{g'}$ and $\mathcal S_{g'}$ whenever they occur, 
this follows from the positivity of the Schur complement established in the proof of Theorem~\ref{thm:core_bellman}.
For the fractional complete-split function $C_{g',d}(\lambda)$, 
the pole bounds \eqref{eq:core_cs_pole} at $(g',d)$ and $(g,d+1)$,
together with the closure identity \eqref{eq:core_cs_closure}, 
verify the hypotheses of Lemma~\ref{lem:transfer_branch}.
Consequently, $C_{g',d}(\lambda)<\lambda^2-d$ as well,
and all transfer comparisons in the three cases below take place on the same increasing branch.

There are three cases for $d$.

\emph{First, suppose that $d\le k_{g'}+1$.}
Since $(g',d)$ is feasible, $\mathcal R_{g'}$ fits within the cap $d$.
Every admissible inner core has at most $d\le k_{g'}+1$ vertices,
so Theorem~\ref{thm:core_dense_all} gives
\[
\Phi(g',d;\lambda)=W_{\mathcal R_{g'}}(\lambda).
\]
By the increasing-branch property,
\[
W_H(\lambda)
\le\mathscr F_\lambda(d,W_{\mathcal R_{g'}}(\lambda)).
\]

\emph{Second, suppose that $k_{g'}+2\le d\le g'$.}
In this case $g'\ge4$ and the integer cap $d$ lies
in the range of the induction hypothesis~\eqref{eq:core_cs_capped}.
Hence
\[
\Phi(g',d;\lambda)
\le\max\{W_{\mathcal R_{g'}}(\lambda),C_{g',d}(\lambda)\}.
\]
Since all three arguments are below the transfer pole,
we may apply the increasing transfer to this inequality.
Using the closure identity~\eqref{eq:core_cs_closure} for the fractional term,
we obtain
\[
\begin{aligned}
W_H(\lambda)
&\le\max\{\mathscr F_\lambda(d,W_{\mathcal R_{g'}}(\lambda)),
           \mathscr F_\lambda(d,C_{g',d}(\lambda))\}\\
&=\max\{\mathscr F_\lambda(d,W_{\mathcal R_{g'}}(\lambda)),
          C_{g,d+1}(\lambda)\}.
\end{aligned}
\]

\emph{Finally, suppose that $d\ge g'+1$.}
Every core with $g'$ edges has at most $g'+1$ vertices,
so the inner cap $d$ is no longer restrictive.
For $g'\ge4$, the induction hypothesis at the largest cap $N=g'+1$,
together with $C_{g',g'+1}=W_{\mathcal S_{g'}}$ from~\eqref{eq:core_cs_walk},
gives
\[
\Phi(g',d;\lambda)
\le\max\{W_{\mathcal R_{g'}}(\lambda),W_{\mathcal S_{g'}}(\lambda)\}.
\]
For $g'\le3$ the same inequality follows from the base cases.
Because $d\ge g'+1$ is equivalent to $d\ge(g+1)/2$,
and we already have $d\ge k_g+1$, condition (K2) applies:
\[
\mathscr F_\lambda(d,W_{\mathcal S_{g'}}(\lambda))
\le C_{g,d+1}(\lambda).
\]
Applying the increasing transfer to the inner bound therefore gives
\[
W_H(\lambda)
\le\max\{\mathscr F_\lambda(d,W_{\mathcal R_{g'}}(\lambda)),
          C_{g,d+1}(\lambda)\}.
\]

The three cases establish the first-part estimate~\eqref{eq:core_firstpart_strong}. By (K1), whose order
hypothesis was verified above,
\[
W_H(\lambda)
\le\max\{W_{\mathcal R_g}(\lambda),C_{g,d+1}(\lambda)\}.
\]
Finally, $k_g+1\le d\le g-1$ and $d+1\le N$ imply $k_g+2\le d+1\le N$.
We may therefore apply the intermediate-cap estimate~\eqref{eq:core_cs_prefix} at $N'=d+1$, 
obtaining
\[
W_H(\lambda)
\le\max\{W_{\mathcal R_g}(\lambda),C_{g,N}(\lambda)\}.
\]
This bound also holds in the two cases $d\le k_g$ and $d=g$ treated earlier.
Taking the maximum over all admissible $H$ proves~\eqref{eq:core_cs_capped}
for the present $g$ and $N$, and completes the induction.

At $N=g+1$ the order cap includes every $g$-edge threshold core, and $C_{g,g+1}=W_{\mathcal S_g}$.
Thus~\eqref{eq:core_cs_capped} gives the upper inequality in~\eqref{eq:core_pointwise_conjecture};
the reverse inequality follows because $\mathcal R_g$ and $\mathcal S_g$ are both admissible.
The direct base cases extend this identity to $g=1,2,3$,
completing the proof of the pointwise identity~\eqref{eq:core_pointwise_conjecture} for every $1\le g\le G$.
\end{proof}

Theorem~\ref{thm:core_cs_conditional} reduces the pointwise comparison to (J), (K1), and (K2), which we
verify next.

\section{The three uniform inequalities}
\label{subsec:core_uniform_checks}

We prove (K2), (J), and (K1), in that order, for every integer $g\ge4$ and $\lambda>\rho(\mathcal R_g)$.
By Theorem~\ref{thm:core_cs_conditional}, these inequalities imply the capped
bound~\eqref{eq:core_cs_capped} and the pointwise comparison~\eqref{eq:core_pointwise_conjecture}.

The proof of (K2) uses a rational identity and a quadratic estimate, with exact checks for a few small
cases. For (J) we use coefficient signs and Sturm root counts, and for (K1) coefficient-positivity
certificates after substitutions into nonnegative variables. All the certificates apply to the stated
parameter ranges, rather than to selected numerical instances.

We use the rational expression~\eqref{eq:WR} for $W_{\mathcal R_g}(\lambda)$, whose denominator
$\DR{g}(\lambda)$ is positive for $\lambda>\rho(\mathcal R_g)$.
For the fractional complete-split function~\eqref{eq:core_cs_walk}, 
we write
\[
\Delta_C=\lambda^2-(s-1)\lambda-s(N-s),
\qquad s=s(g,N).
\]
The parameters $g,N,s$ in $\Delta_C$ are those of the function being considered.
By~\eqref{eq:core_cs_pole}, $\Delta_C>0$ in the permitted cap interval.

The certificates use exact symbolic arithmetic; 
the corresponding scripts are identified in the proofs and Appendix~\ref{app:section6_runners}.

\subsection{Verification of (K2)}\label{subsubsec:proof_K2}

\begin{Th}\label{thm:K2}
Let $g\ge4$ and let $d$ be an integer satisfying $(g+1)/2\le d\le g-1$. 
Let $g'=g-d$ and suppose that $\lambda>\rho(\mathcal R_g)$.
If $d\le k_g$, then
\[
 \mathscr F_\lambda(d,W_{\mathcal S_{g'}}(\lambda))
 \le W_{\mathcal R_g}(\lambda).
\]
If $d\ge k_g+1$, then
\[
 \mathscr F_\lambda(d,W_{\mathcal S_{g'}}(\lambda))
 \le C_{g,d+1}(\lambda),
\]
with equality in the latter inequality if and only if $d=(g+1)/2$. 
Consequently, condition (K2) holds for every $g\ge4$.
\end{Th}

\begin{proof}
Since $d\ge(g+1)/2$, we have $d\ge g'+1$ for $g'=g-d$. 
Hence
\[
V_{d+1,g'}=K_1\vee[\mathcal S_{g'}\cup(d-g'-1)K_1]
\]
is a well-defined core with $g$ edges. 
By the transfer identity~\eqref{eq:core_bellman_identity},
\[
\mathscr F_\lambda(d,W_{\mathcal S_{g'}}(\lambda))=W_{V_{d+1,g'}}(\lambda).
\]
If $d\le k_g$, this core has $d+1\le k_g+1$ vertices, and Theorem~\ref{thm:core_dense_all} 
gives
\[
W_{V_{d+1,g'}}(\lambda)\le W_{\mathcal R_g}(\lambda).
\]
Equality occurs for $(g,d)=(4,3)$ and $(5,3)$, when $V_{d+1,g'}\cong\mathcal R_g$.

Next, suppose that $d\ge k_g+1$, and set $N=d+1$.
Then $k_g+2\le N\le g$ and $g=(N-1)+g'$.
From $(g+1)/2\leq d\leq g-1$ we have $1\le g'\le N-2$, 
hence
\[
N-1<g\le2N-3.
\]
For $s=s(g,N)$, the edge equation~\eqref{eq:core_cs_s} is
\[
E_N(s)=\binom{s}{2}+s(N-s)=g.
\]
Because $N\ge k_g+2\ge5$,
\[
E_N(1)=N-1,\qquad E_N(2)=2N-3,\qquad
E_N'(s)=N-\frac12-s>0\quad(1\le s\le2).
\]
Since $N-1<g\le2N-3$, the equation $E_N(s)=g$ has a unique solution in $(1,2]$, 
namely the smaller root $s=s(g,N)$ in~\eqref{eq:core_cs_s}.
Equality $s=2$ occurs precisely when $g=2N-3$, that is, when $d=(g+1)/2$.

Using the Bellman transfer~\eqref{eq:core_bellman_transfer}, 
the fractional function~\eqref{eq:core_cs_walk} and $E_N(s)=g$, 
we obtain
\begin{equation}\label{eq:K2id}
 C_{g,N}(\lambda)-\mathscr F_\lambda\bigl(N-1,W_{\mathcal S_{g'}}(\lambda)\bigr)
 =\frac{\lambda(\lambda+1)(2-s)(s-1)\,Q_{\mathrm{K2}}}{4\Delta_C\,P_X},
\end{equation}
where
\[
\begin{aligned}
P_X&=(\lambda^2-g')\bigl(\lambda^2-(N-1)-W_{\mathcal S_{g'}}(\lambda)\bigr), \\
Q_{\mathrm{K2}}&=a_2\Delta_C+c_1\lambda+c_0, \\
a_2&=4N^2-4Ns-16N+s^2+9s+12,\\
c_1&=-s(4N-3s-3),\\
c_0&=(N-s)[4N^2-10N(s+1)+5s^2+11s+6].
\end{aligned}
\]

The denominator in~\eqref{eq:K2id} is positive for $\lambda>\rho(\mathcal R_g)$. 
Indeed, by Rowlinson's theorem~\cite{rowlinson1988},
\[
\rho(V_{d+1,g'})\le\rho(\mathcal R_g)<\lambda.
\]
Since $V_{d+1,g'}$ properly contains $\mathcal S_{g'}$,
we have $\rho(V_{d+1,g'})>\sqrt{g'}$, and hence $\lambda^2-g'>0$. 
The positive Schur complement in the proof of Theorem~\ref{thm:core_bellman}, applied to $V_{d+1,g'}$, 
also gives
\[
\lambda^2-d-W_{\mathcal S_{g'}}(\lambda)>0.
\]
Thus $P_X>0$, while $\Delta_C>0$ by~\eqref{eq:core_cs_pole} at $N=d+1\in[k_g+2,g]$. 
Since $1<s\le2$, the sign of~\eqref{eq:K2id} for $s<2$ is the sign of $Q_{\mathrm{K2}}(\lambda)$, 
and the difference vanishes when $s=2$.

To prove $Q_{\mathrm{K2}}(\lambda)>0$ for $\lambda>\rho(\mathcal R_g)$, 
we use its constant term and leading coefficient:
\[
Q_{\mathrm{K2}}(0)=-(N-s)(s-1)(2N-s-2)(2N-s-3)<0
\]
and
\[
a_2=4N^2-24N+22
    +4N(2-s)+(s-1)(s+10)>0.
\]
Both inequalities follow from $N\ge5$ and $1<s\le2$; in particular, $4N^2-24N+22\ge2$. 
Thus $Q_{\mathrm{K2}}$ has exactly one positive root and is positive to its right. 
Since $r=r_{g,N}<\rho(\mathcal R_g)$ by~\eqref{eq:core_cs_pole}, it suffices to prove $Q_{\mathrm{K2}}(r)>0$.
Since $\Delta_C(r)=0$, we have
\[
Q_{\mathrm{K2}}(r)=c_0+c_1r=c_0-|c_1|r.
\]
We treat $N\ge7$ analytically and the remaining integer pairs by exact checks.

\emph{First suppose that $N\ge7$.}
Write
\[
H(s)=4N^2-10N(s+1)+5s^2+11s+6,
\]
so that $c_0=(N-s)H(s)$. Since
\[
H'(s)=-10N+10s+11<0\qquad(1\le s\le2),
\]
we have $H(s)\ge H(2)=4N^2-30N+48>0$, and hence
\[
c_0\ge(N-2)(4N^2-30N+48).
\]
Similarly, $|c_1|=s(4N-3s-3)$ is increasing on $[1,2]$, as $4N-6s-3>0$, and therefore 
\[
|c_1|\le2(4N-9).
\]

To bound the positive root $r$ of $\Delta_C$, note that $s(2N-3-s)$ is also increasing on $[1,2]$. 
Consequently,
\[
\begin{aligned}
\Delta_C(N-3)
&=(N-3)(N-2)-s(2N-3-s)\\
&\ge (N-3)(N-2)-2(2N-5)\\
&=N^2-9N+16>0.
\end{aligned}
\]
Since $\Delta_C(0)<0$ and $r$ is its only positive zero, we obtain $r<N-3$. 
Combining the preceding bounds gives
\[
\begin{aligned}
Q_{\mathrm{K2}}(r)
&=c_0-|c_1|r\\
&>(N-2)(4N^2-30N+48)-2(4N-9)(N-3)\\
&=2(2N-5)(N^2-9N+15)>0,
\end{aligned}
\]
where the last inequality follows from $N^2-9N+15=(N-7)^2+5(N-7)+1>0$ for $N\ge7$.

\emph{Now suppose that $N\le6$.}
Since $N=d+1$, the restrictions $k_g+2\le N\le6$ and $N-1<g\le2N-3$ 
leave exactly five admissible integer pairs:
\[
(g,d)=(5,4),(6,5),(7,5),(8,5),(9,5).
\]
The last pair has $s=2$ and gives equality in~\eqref{eq:K2id}.
For each of the four strict cases, 
exact Sturm root counts for $\DR{g}$ provide a rational lower bound $\ell\le\rho(\mathcal R_g)$.
Writing
\[
Q_{\mathrm{K2}}(\lambda)=a_2\lambda^2+a_1\lambda+a_0,
\]
we verify
\[
a_2>0,\qquad
Q_{\mathrm{K2}}(\ell)>0,\qquad
Q_{\mathrm{K2}}'(\ell)>0
\]
using exact arithmetic in the quadratic field containing $s=s(g,N)$. Since
\[
Q_{\mathrm{K2}}'(\lambda)
=Q_{\mathrm{K2}}'(\ell)+2a_2(\lambda-\ell)>0
\qquad(\lambda\ge\ell),
\]
the polynomial $Q_{\mathrm{K2}}$ is strictly increasing on $[\ell,\infty)$, and hence
\[
Q_{\mathrm{K2}}(\lambda)\ge Q_{\mathrm{K2}}(\ell)>0
\qquad(\lambda\ge\ell).
\]
This proves the required positivity for every $\lambda>\rho(\mathcal R_g)$.
The identity~\eqref{eq:K2id} and the finite checks for $N\le6$ are implemented 
in \path{python/k2_proof_checks.py} and \path{sage/sage_J_K2.py}; 
see Appendix~\ref{app:K2_verification}.

For $1<s<2$, all factors on the right of~\eqref{eq:K2id} are positive. 
For $s=2$, equivalently $d=(g+1)/2$, the difference vanishes. 
Together with the case $d\le k_g$, this proves (K2) and its equality assertion.
\end{proof}

\subsection{Verification of (J)}\label{subsubsec:proof_J}

We first consider the dense-end inequality
\begin{equation}\label{eq:core_J_dense_end}
C_{g,k+2}(\lambda)<W_{\mathcal R_g}(\lambda),
\qquad \lambda>\rho(\mathcal R_g).
\end{equation}
For the edge counts where this inequality does not hold throughout the spectral interval, 
we instead establish the required monotonicity of $C_{g,N}(\lambda)$.
In fact, for the nondegenerate exceptional cases 
we prove strict increase on the entire real interval $[k+2,g+1]$.

\begin{Th}\label{thm:J}
Let $g\ge4$ with $\displaystyle g=\binom{k}{2}+t$, where $0\le t<k$.
Set
$
 E=\{4,5,6,7,8,10,11,12,13,16,17\}.
$
\begin{enumerate}[label=(\roman*)]
\item If $g\notin E$, then 
the dense-end inequality~\eqref{eq:core_J_dense_end} holds for every $\lambda>\rho(\mathcal R_g)$.

\item If $g\in E\setminus\{4\}$, let $\lambda_g^*$ be given by
\[
\begin{array}{c|rrrrrrrrrr}
 g            &5&6&7&8&10&11&12&13&16&17\\ \hline
 \lambda_g^*  &8&9&10&11&50&20&20&60&300&70
\end{array}
\]
Then \eqref{eq:core_J_dense_end} holds whenever $\rho(\mathcal R_g)<\lambda\le\lambda_g^*$.
For $\lambda\ge\lambda_g^*$, the function $N\mapsto C_{g,N}(\lambda)$ is strictly increasing on $[k+2,g+1]$.

\item For $g=4$, the interval of permitted order caps consists of the single point $k+2=g+1=5$.
\end{enumerate}
Consequently, condition (J) holds for every $g\ge4$.
\end{Th}

The rational values $\lambda_g^*$ are cutoffs chosen so that the certified ranges for the dense-end
inequality and cap monotonicity meet. They are not spectral transition points, and no claim of optimality
is made.

\begin{proof}
Write $h=k-t$ and $s(g,k+2)=k-u$.
By~\eqref{eq:core_cs_u_relation}, we have
\[
u>0,\qquad u^2+3u=2(k+h).
\]
We first consider the dense-end difference in~\eqref{eq:core_J_dense_end}.
Insert $N=k+2$ and $s=k-u$ into the formulas
\eqref{eq:WR} and~\eqref{eq:core_cs_walk},
and put $x=\lambda-(k-1)$. Simplification yields
\begin{equation}\label{eq:core_J_numerator}
 W_{\mathcal R_g}(\lambda)-C_{g,k+2}(\lambda)
 =\frac{\mathcal N_{\mathrm J}(x)}{8\Delta_C\,\DR{g}(\lambda)},
\end{equation}
where
\[
 \mathcal N_{\mathrm J}(x)=\sum_{i=0}^3c_i(u,h)x^i
\]
and the four coefficient polynomials $c_i(u,h)$ are:
\begingroup\small
\begin{align*}
4c_0={}&-64h^5+128h^4u^2+448h^4u-104h^3u^4-688h^3u^3-1160h^3u^2+32h^3u-192h^3\\
&+44h^2u^6+404h^2u^5+1268h^2u^4+1356h^2u^3+32h^2u^2+416h^2u\\
&-10hu^8-112hu^7-476hu^6-928hu^5-738hu^4-160hu^3-264hu^2\\
&+u^{10}+13u^9+66u^8+166u^7+221u^6+169u^5+96u^4+36u^3,\\
2c_1={}&96h^4-144h^3u^2-496h^3u+64h^3+84h^2u^4+520h^2u^3+788h^2u^2-240h^2u+96h^2\\
&-24hu^6-192hu^5-516hu^4-424hu^3+132hu^2-128hu\\
&+3u^8+28u^7+92u^6+126u^5+73u^4+38u^3+24u^2,\\
c_2={}&-48h^3+48h^2u^2+160h^2u-48h^2-18hu^4-92hu^3-110hu^2+76hu-16h\\
&+3u^6+17u^5+25u^4+3u^3-4u^2+4u,\\
c_3={}&2\bigl(8h^2-4hu^2-12hu+8h+u^4+2u^3-u^2-2u\bigr).
\end{align*}
\endgroup
The coefficients were obtained by eliminating $k$ through~\eqref{eq:core_cs_u_relation}:
\[
 k=\frac{u^2+3u}{2}-h,
 \qquad s=k-u.
\]
For example, when $t=0$ we have $h=k$ and therefore
$h=(u^2+3u)/4$.
In this case,
\[
 c_3=u(u^3-2u^2-7u+8),
\]
which is positive for $u\ge4$ but negative at $u=3$.
This shows why small $u$ requires a separate check.

The denominators in~\eqref{eq:core_J_numerator} are positive:
$\Delta_C>0$ by the pole bound~\eqref{eq:core_cs_pole}
at the cap $N=k+2$, and $\DR{g}(\lambda)>0$
by~\eqref{eq:WR}, since $\lambda>\rho(\mathcal R_g)$.
Moreover, $\mathcal R_g$ contains $K_k$, so
$\rho(\mathcal R_g)\ge k-1$ and $x=\lambda-(k-1)>0$.
Hence~\eqref{eq:core_J_dense_end} follows if $c_0>0$ and $c_1,c_2,c_3\ge0$. We check these signs
separately for $k+h\ge14$ and the nonexceptional cases with $k+h\le13$.

\emph{The range $k+h\ge14$.}
By~\eqref{eq:core_cs_u_relation}, $u>0$ satisfies
$u^2+3u=2(k+h)$. The left-hand side is strictly
increasing in $u>0$ and equals $28$ at $u=4$;
thus $k+h\ge14$ is equivalent to $u\ge4$.
Let
\[
 u=4+y,
 \qquad h=\frac{q(u^2+3u)}4,
 \qquad q=\frac{2h}{k+h}.
\]
Here $y=u-4\ge0$, and $1\le h\le k$ gives
$0<q=2h/(k+h)\le1$.
For each $i=0,1,2,3$, expand
\[
 c_i\left(4+y,\frac{q[(4+y)^2+3(4+y)]}{4}\right)
 =\sum_j a_{ij}(q)y^j,
\]
where every $a_{ij}(q)$ is a polynomial with rational coefficients.
For each coefficient polynomial $a_{ij}(q)$, an exact Sturm
root count gives no zero in $[0,1]$, while exact evaluation
at $q=1/2$ gives a positive value.
Continuity then implies $a_{ij}(q)>0$ throughout $[0,1]$.
Because $y\ge0$, all four $c_i(u,h)$ are strictly positive.
Thus $\mathcal N_{\mathrm J}(x)>0$ for $x>0$ in
\eqref{eq:core_J_numerator}, proving the dense-end inequality
for every $k+h\ge14$.
The exact Sturm tests cover the full region $u\ge4$, $0\le q\le1$.

\emph{The range $k+h\le13$.}
The conditions $k\ge3$, $1\le h\le k$, $\displaystyle g=\binom{k}{2}+k-h\ge4$, and $k+h\le13$
leave exactly $38$ integer pairs $(k,h)$.
For each pair,
\[
u=\frac{-3+\sqrt{9+8(k+h)}}2.
\]
Using~\eqref{eq:core_cs_u_relation}, 
each coefficient $c_i(u,h)$ reduces to a linear expression in $u$,
whose sign can be determined exactly in $\mathbb Q(\sqrt{9+8(k+h)})$.

For the following $27$ edge counts,
\[
 g\in\{9,14,15,18,19,20,22,\ldots,27,
       31,\ldots,35,41,\ldots,44,52,53,54,64,65,77\},
\]
all four coefficients are nonnegative, with $c_0>0$.
More precisely, $c_3=0$ for $g=9,18$, while all the other coefficients are positive.
Thus $\mathcal N_{\mathrm J}(x)>0$ for $x\ge0$, proving~\eqref{eq:core_J_dense_end} for these $27$ edge
counts by~\eqref{eq:core_J_numerator}.

The remaining $11$ edge counts form precisely the exceptional
set $E$ in Theorem~\ref{thm:J}.
The exact sign pattern of $(c_0,c_1,c_2,c_3)$ is
$(+,+,+,-)$ for $g\in E\setminus\{4\}$ and
$(+,+,-,-)$ for $g=4$.
The corresponding exact quadratic-field tests are listed in
Appendix~\ref{app:J_verification}.
We treat $g=4$ separately at the end, since its permitted
cap interval is a single point.

\emph{The exceptional edge counts $g\in E\setminus\{4\}$.}
For each $g\in E\setminus\{4\}$, the cubic
$\mathcal N_{\mathrm J}(x)=c_0+c_1x+c_2x^2+c_3x^3$
has coefficient signs $(+,+,+,-)$ in ascending powers.
Thus there is exactly one sign change in descending powers,
and Descartes' rule gives at most one positive root.
Since the constant term is positive and the leading coefficient
is negative, the polynomial also changes sign between $0$
and $+\infty$. Hence it has exactly one positive root.
The exact sign evaluations give
\[
 \mathcal N_{\mathrm J}(\lambda_g^*-(k-1))>0
\]
for the ten values $\lambda_g^*$ listed in part~(ii).
Consequently, the unique positive root lies strictly to the right of $\lambda_g^*-(k-1)$,
so $\mathcal N_{\mathrm J}(x)>0$ whenever $0\le x\le\lambda_g^*-(k-1)$.
Since $x=\lambda-(k-1)$, this proves~\eqref{eq:core_J_dense_end} for $\rho(\mathcal R_g)<\lambda\le\lambda_g^*$.

For $g\in E\setminus\{4\}$ and $\lambda\ge\lambda_g^*$,
it remains to show that $C_{g,N}(\lambda)$ is strictly increasing in $N$ on $[k+2,g+1]$.
Write $s=s(g,N)$ and $\tau=N-s$.
The derivative calculation in the proof of Theorem~\ref{thm:core_cs_valley}, 
together with~\eqref{eq:core_cs_s_derivative}, 
gives
\[
\frac{d}{dN}C_{g,N}(\lambda)
=-\frac{s\eta(s)}{(2\tau-1)\Delta_C^2}
=\frac{H_g(N,\lambda,s)}
       {(2\tau-1)\Delta_C^2},
\]
where
\[
H_g(N,\lambda,s)
=s\left[
2\tau(\tau-1)\lambda(\lambda+1)
-s(2\tau-1)(\lambda+\tau)^2
\right].
\]
Since $\tau>1$ and $\Delta_C>0$ by Theorem~\ref{thm:core_cs_valley},
the derivative has the sign of $H_g$.

We first verify $H_g>0$ at the left endpoint $N=k+2$
for every $\lambda\ge\lambda_g^*$.
At this fixed cap, the edge equation~\eqref{eq:core_cs_s}
is equivalent to
\[
 s^2-(2N-1)s+2g=0,
 \qquad
 s=N-\frac12-\sqrt{\mathfrak D_g},
 \qquad \mathfrak D_g=\left(N-\frac12\right)^2-2g>0.
\]
For each fixed exceptional $g$, reduce $H_g(k+2,\lambda,s)$ modulo $s^2-(2N-1)s+2g$, with $N=k+2$.
The remainder is linear in $s$, so substituting $s=N-\frac12-\sqrt{\mathfrak D_g}$ gives
\[
H_g(k+2,\lambda,s)
=A_g(\lambda)-B_g(\lambda)\sqrt{\mathfrak D_g},
\]
where $A_g,B_g$ are polynomials in $\lambda$ with rational coefficients.
To determine its sign for $\lambda\ge\lambda_g^*$, consider the conjugate-product polynomial
\[
 \mathcal P_g(\lambda)
 =A_g(\lambda)^2-\mathfrak D_g B_g(\lambda)^2.
\]
An exact Sturm calculation verifies that $\mathcal P_g$ has no real root in $[\lambda_g^*,\infty)$.
An exact quadratic-irrational sign comparison also gives
\[
 A_g(\lambda_g^*)-B_g(\lambda_g^*)\sqrt{\mathfrak D_g}>0.
\]
A zero of $A_g-B_g\sqrt{\mathfrak D_g}$ would also be a zero of $\mathcal P_g$. The Sturm result and the
positive endpoint value therefore imply, by continuity, that $A_g-B_g\sqrt{\mathfrak D_g}>0$ for every
$\lambda\ge\lambda_g^*$.
We conclude that
\[
 \left.\frac{d}{dN}C_{g,N}(\lambda)\right|_{N=k+2}>0
 \qquad(\lambda\ge\lambda_g^*).
\]

By Theorem~\ref{thm:core_cs_valley}, 
the function $N\mapsto C_{g,N}(\lambda)$ is 
either strictly monotone 
or decreases to a unique minimum and then increases.
Since its derivative is positive at $N=k+2$,
it must be strictly increasing throughout $[k+2,g+1]$, proving part~(ii).

The exact coefficient signs, algebraic endpoint comparisons,
and Sturm root counts used above are checked by
\path{python/j_proof_checks.py} and \path{sage/sage_J_K2.py};
see Appendix~\ref{app:J_verification} for the individual certificates.

Finally, (J) holds trivially for $g=4$, since $k_g+2=g+1=5$. 
For $g\notin E$, and for $g\in E\setminus\{4\}$ with $\lambda\le\lambda_g^*$,
parts~(i) and~(ii) give
\[
C_{g,k_g+2}(\lambda)<W_{\mathcal R_g}(\lambda),
\]
so the hypothesis of (J) cannot hold.
For $g\in E\setminus\{4\}$ with $\lambda>\lambda_g^*$, 
part~(ii) shows that $C_{g,N}(\lambda)$ is strictly increasing in $N$ on $[k_g+2,g+1]$, 
which implies the required integer-cap monotonicity. 
Hence (J) holds for all $g\ge4$ and $\lambda>\rho(\mathcal R_g)$.
\end{proof}

\paragraph{Strict endpoint comparison.}
Combining Theorems \ref{thm:core_cs_valley} and \ref{thm:J}, 
we obtain a strict bound for the fractional complete-split walk function $C_{g,N}(\lambda)$, 
which will be used in Theorem~\ref{thm:strict_pointwise}.

\begin{Lemma}[Strict endpoint bound]
\label{lem:strict_prefix}
Let $g\ge5$ and $\lambda>\rho(\mathcal R_g)$.
For every real $N$ with $k_g+2\le N\le g$, we have
\[
C_{g,N}(\lambda)
<
\max\{W_{\mathcal R_g}(\lambda),
      C_{g,g+1}(\lambda)\}.
\]
\end{Lemma}

\begin{proof}
By Theorem~\ref{thm:J}, either
\[
C_{g,k_g+2}(\lambda)<W_{\mathcal R_g}(\lambda)
\]
or $C_{g,N}(\lambda)$ is strictly increasing in $N$ on $[k_g+2,g+1]$.

In the latter case, $N\le g<g+1$ gives
\[
C_{g,N}(\lambda)<C_{g,g+1}(\lambda).
\]

In the former case, the assertion is immediate for $N=k_g+2$. 
For $k_g+2<N\le g$, Theorem \ref{thm:core_cs_valley} gives
\[
C_{g,N}(\lambda)
<\max\{C_{g,k_g+2}(\lambda),
        C_{g,g+1}(\lambda)\}
\le\max\{W_{\mathcal R_g}(\lambda),
          C_{g,g+1}(\lambda)\}.
\qedhere          
\]
\end{proof}

\subsection{Verification of (K1)}\label{subsubsec:proof_K1}

Condition (K1) of Theorem~\ref{thm:core_cs_conditional} compares 
the transfer of $W_{\mathcal R_{g'}}(\lambda)$, where $g'=g-d$, 
with the dense and fractional bounds for $g$. 
For
\[
k_g+1\le d\le g-1,
\qquad |V(\mathcal R_{g'})|\le d,
\]
we must prove, for every $\lambda>\rho(\mathcal R_g)$, that
\[
\mathscr F_\lambda(d,W_{\mathcal R_{g'}}(\lambda))
\le\max\{W_{\mathcal R_g}(\lambda),C_{g,d+1}(\lambda)\}.
\]

\paragraph{Reparametrization and comparison numerator.}
Let $(g,d)$ be an admissible pair in (K1), and let $g'=g-d\ge1$.
Write
\[
g'=\binom{k'}2+t',
\qquad 0\le t'<k',
\]
and introduce
\[
a=k'+1,\qquad b=t'+1,\qquad p=d-a.
\]
Thus $a\ge3$ and $1\le b\le a-1$.
The definitions give
\[
d=a+p,\qquad
 g=\binom a2+b+p.
\]

If $t'>0$, then $|V(\mathcal R_{g'})|=a$,
so the feasibility hypothesis of (K1) gives $p=d-a\ge0$.
If $t'=0$, then $\mathcal R_{g'}=K_{a-1}$,
and feasibility gives $p\ge-1$.
Suppose in either case that $p\le0$.
Since $1\le b\le a-1$ and $p\ge-1$,
\[
\binom a2\le g=\binom a2+b+p<\binom{a+1}2.
\]
Hence $k_g=a$, whereas $d=a+p\le a$, contrary to the hypothesis $d\ge k_g+1$.
It follows that every admissible instance of (K1) has $p\ge1$.

Let $Y(a,b,p)$ be the graph obtained from a clique $K_a$ with a fixed vertex $v$ 
by adding a vertex $x$ adjacent to $v$ and $b-1$ other clique vertices,
and attaching $p$ pendant vertices to $v$.
Thus $Y$ has $a+p+1=d+1$ vertices and $\binom a2+b+p=g$ edges.
For $t'>0$, the vertex $v$ is the new universal vertex,
while $x$ is the nonclique vertex of $\mathcal R_{g'}$.
For $t'=0$, we have $b=1$ and $x$ is one of the $p+1$ isolated vertices made adjacent to $v$.
In both cases, $Y$ is the cone over $\mathcal R_{g'}\cup(d-|V(\mathcal R_{g'})|)K_1$.
Hence the transfer identity~\eqref{eq:core_bellman_identity} gives
\begin{equation}\label{eq:core_K1_Y_transfer}
W_Y(\lambda)
=\mathscr F_\lambda(d,W_{\mathcal R_{g'}}(\lambda)).
\end{equation}

Substituting~\eqref{eq:WR} into the Bellman transfer \eqref{eq:core_bellman_transfer}, 
we obtain 
\[
W_Y(\lambda)=\frac{n_Y}{P_Y},
\]
where
\begin{align*}
n_Y&=\lambda(\lambda+1)
 \bigl[(\lambda+d)\DR{g'}(\lambda)
       +\nR{g'}(\lambda)\bigr]
 -(\lambda+d+1)P_Y,\\
P_Y&=(\lambda^2-d)\DR{g'}(\lambda)
      -\nR{g'}(\lambda).
\end{align*}
Since $Y$ has $g$ edges, Rowlinson's theorem~\cite{rowlinson1988} gives $\rho(Y)\le\rho(\mathcal R_g)<\lambda$.
The positive Schur complement in the proof of Theorem~\ref{thm:core_bellman} therefore yields
\[
\lambda^2-d-W_{\mathcal R_{g'}}(\lambda)>0.
\]
Since $\DR{g'}(\lambda)>0$ by~\eqref{eq:WR}, we conclude that
\[
P_Y=\DR{g'}(\lambda)
[\lambda^2-d-W_{\mathcal R_{g'}}(\lambda)]>0.
\]
Also $\DR{g}(\lambda)>0$ by~\eqref{eq:WR}. Thus $W_{\mathcal R_g}-W_Y$ has the sign of
\begin{equation}\label{eq:core_K1_dense_num}
\mathcal N_{\mathrm A}
:=\nR{g}(\lambda)P_Y
   -n_Y\DR{g}(\lambda)
=[W_{\mathcal R_g}(\lambda)-W_Y(\lambda)]
   \DR{g}(\lambda)P_Y.
\end{equation}
In particular, $\mathcal N_{\mathrm A}>0$ is equivalent to $W_Y<W_{\mathcal R_g}$.

\paragraph{The comparison polynomial.}
The polynomial $\mathcal N_{\mathrm A}$ is obtained from~\eqref{eq:WR}, \eqref{eq:core_K1_dense_num} and 
the expressions for $P_Y,n_Y$, using parameters $(k,h)$ for $\mathcal R_g$ and $(a-1,a-b)$ for $\mathcal R_{g'}$. 
Before imposing the edge-count relation
\[
\binom{k}{2}+t=\binom a2+b+p,
\]
its expansion contains $374$ monomials and has degree $7$ in $\lambda$. 
In the proofs below, we use this relation to eliminate one parameter before checking coefficient positivity.

\begin{Th}\label{thm:K1}
Every instance of (K1) satisfies, for every $\lambda>\rho(\mathcal R_g)$, the following comparisons:
\begin{enumerate}[label=(\alph*)]
\item If $p\le a-1$, then $W_Y(\lambda)<W_{\mathcal R_g}(\lambda)$.
\item If $a=3$, then $W_Y(\lambda)\le C_{g,d+1}(\lambda)$.
\item If $p\ge a$ and $a\ge4$, then
\[
W_Y(\lambda)<W_{\mathcal R_g}(\lambda)
\quad\text{for }\lambda\le a+\frac{2a^2}{p},
\]
and
\[
W_Y(\lambda)\le C_{g,d+1}(\lambda)
\quad\text{for }\lambda\ge a+\frac{2a^2}{p}.
\]
\end{enumerate}
In particular, condition (K1) holds for every $g\ge4$.
\end{Th}

\paragraph{Proof of part (a): $1\le p\le a-1$.}
If $b+p<a$, then $k_g=a$, $t=b+p$ and
$h=a-b-p$. We make the substitutions
\[
b=1+B,\qquad p=1+P,\qquad
a=3+B+P+m,\qquad \lambda=k_g-1+x.
\]
Here $B=b-1$, $P=p-1$ and $m=a-b-p-1$
are nonnegative, so these substitutions cover
the first case.

If $b+p\ge a$, then
\[
g=\binom{a+1}{2}+(b+p-a),
\]
where $0\le b+p-a<a+1$. Hence $k_g=a+1$ and
$h=2a+1-b-p$. In this case we put
\[
b=1+r+c,\qquad a=2+r+c+m,\qquad
p=1+m+r,\qquad \lambda=k_g-1+x.
\]
Conversely, the inverse relations are
\[
r=b+p-a,\qquad
c=a-p-1,\qquad
m=a-b-1,
\]
and each is nonnegative by $b+p\ge a$, $p\le a-1$, and $b\le a-1$, respectively.

In both ranges, $x=\lambda-(k_g-1)>0$ 
because $\mathcal R_g$ contains $K_{k_g}$ and $\lambda>\rho(\mathcal R_g)\ge k_g-1$.
Both substitutions leave only nonnegative variables.
The resulting polynomials have $431$ and $470$ nonzero monomials, respectively, 
all with positive integer coefficients. 
Hence $\mathcal N_{\mathrm A}>0$ in both ranges.
Since $\DR{g}P_Y>0$ in~\eqref{eq:core_K1_dense_num}, we obtain
\[
W_Y(\lambda)<W_{\mathcal R_g}(\lambda).
\]
The two expansions and their coefficient signs are checked by
\path{python/k1_mregion_certificate.py} and \path{sage/sage_k1.py};
see Appendix~\ref{app:K1a_verification}.
This proves part~(a).

\paragraph{Proof of part (b): $a=3$.}
For $a=3$, we have $k'=2$, $g'=b\in\{1,2\}$ and $\mathcal R_{g'}=\mathcal S_{g'}$. 
Moreover, $d\ge k_g+1$ by (K1), and $d=a+p\ge4>g'$ implies $d\ge(g+1)/2$. 
Thus the cap-side inequality of Theorem~\ref{thm:K2} gives
\[
W_Y(\lambda)\le C_{g,d+1}(\lambda),
\]
proving part~(b).

\paragraph{Proof of part (c): $p\ge a$, $a\ge4$.}
We split the spectral range at $\lambda=a+2a^2/p$. 
For smaller $\lambda$, we prove $W_Y<W_{\mathcal R_g}$ 
by establishing positivity of the comparison polynomial $\mathcal N_{\mathrm A}$.
For larger $\lambda$, we first prove $W_{\mathcal R_{g'}}\le C_{g',d}$ and then 
apply the increasing Bellman transfer to obtain $W_Y\le C_{g,d+1}$.
The dividing value is chosen to accommodate the two proofs and 
does not represent a crossing of the walk functions.

\begin{Lemma}[Bracketing the clique-size parameter]
\label{lem:bracket}
Let $X>0$ and $G>0$ satisfy $X^2\ge2G$, and define
\[
\sigma=X-\sqrt{X^2-2G}.
\]
Then
\[
\frac GX\le\sigma\le\frac{GX}{X^2-G}.
\]
\end{Lemma}

\begin{proof}
Multiplying by the conjugate gives
\[
\sigma=\frac{2G}{X+\sqrt{X^2-2G}}.
\]
Because $0\le\sqrt{X^2-2G}\le X$, the denominator is at most $2X$, so $\sigma\ge G/X$.
For the upper bound, observe that both $\sqrt{X^2-2G}$ and $X-2G/X$ are nonnegative, because $X^2\ge2G$.
Moreover,
\[
(X^2-2G)-\left(X-\frac{2G}{X}\right)^2
=2G\left(1-\frac{2G}{X^2}\right)\ge0.
\]
Hence $\sqrt{X^2-2G}\ge X-2G/X$, and therefore
\[
\sigma\le\frac{2G}{2X-2G/X}
=\frac{GX}{X^2-G}.
\]
The denominator $X^2-G$ is positive, so this proves both bounds.
\end{proof}

\begin{Lemma}[Inner dense-core comparison with a fractional bound]
\label{lem:MC}
For the parameters $a,b,p,d,g'$ introduced above, 
suppose that $a\ge4$ and $p\ge a$, and put $\sigma=s(g',d)$. 
Then, for every $\lambda\ge a+2a^2/p$,
\[
\begin{aligned}
\Delta_{C'}(\lambda)
&:=\lambda^2-(\sigma-1)\lambda
   -\sigma(d-\sigma)>0,\\
W_{\mathcal R_{g'}}(\lambda) 
&\le C_{g',d}(\lambda).
\end{aligned}
\]
\end{Lemma}

\begin{proof}
Since $1\le b\le a-1$, the graph $\mathcal R_{g'}$ embeds in $K_a$, so
\[
\rho(\mathcal R_{g'})\le a-1<a\le\lambda.
\]
Thus $\DR{g'}(\lambda)>0$ by~\eqref{eq:WR}. 
The fractional denominator must be checked separately when $C_{g',d}$ does not represent a graph.
From~\eqref{eq:core_cs_walk}, we have
\[
C_{g',d}(\lambda)=\frac{N_{C'}}{\Delta_{C'}},
\qquad
N_{C'}=2g'\lambda+\sigma d(d-\sigma).
\]
To prove the assertion, it remains to show $\Delta_{C'}>0$ and $E\ge0$, 
where
\[
E=N_{C'}\DR{g'}(\lambda)
    -\nR{g'}(\lambda)\Delta_{C'}
 =[C_{g',d}(\lambda)-W_{\mathcal R_{g'}}(\lambda)]
     \Delta_{C'}\DR{g'}(\lambda).
\]
These two inequalities imply $W_{\mathcal R_{g'}}(\lambda)\le C_{g',d}(\lambda)$.

The edge equation~\eqref{eq:core_cs_s} at order parameter $d$
gives
\[
\sigma^2=(2d-1)\sigma-2g'.
\]
Eliminating $\sigma^2$ from $N_{C'}$ and $\Delta_{C'}$
therefore gives
\[
\begin{aligned}
N_{C'}&=2g'(\lambda+d)-d(d-1)\sigma,\\
\Delta_{C'}&=\lambda^2+\lambda-2g'
          -(\lambda+1-d)\sigma.
\end{aligned}
\]
Thus both the denominator $\Delta_{C'}$ and the
comparison numerator $E$ are affine in $\sigma$.
For the latter, write explicitly
\[
\begin{aligned}
E&=U+V\sigma,\\
U&=2g'(\lambda+d)\DR{g'}(\lambda)
   -\nR{g'}(\lambda)(\lambda^2+\lambda-2g'),\\
V&=-d(d-1)\DR{g'}(\lambda)
   +\nR{g'}(\lambda)(\lambda+1-d).
\end{aligned}
\]

Let $X=d-\frac12$. Since $p\ge a$ implies $d\ge2a$, and $g'<\binom a2$, 
we have
\[
X^2\ge\left(2a-\frac12\right)^2>2g'.
\]
As $\sigma=X-\sqrt{X^2-2g'}$, Lemma~\ref{lem:bracket} gives
\[
\sigma_\ell=\frac{g'}X
\le\sigma\le
\frac{g'X}{X^2-g'}=\sigma_u.
\]
Since $E$ and $\Delta_{C'}$ are affine in $\sigma$, their signs on $[\sigma_\ell,\sigma_u]$ follow from
four endpoint inequalities:
\[
E(\sigma_\ell)\ge0,\quad E(\sigma_u)\ge0,
\qquad
\Delta_{C'}(\sigma_\ell)>0,\quad \Delta_{C'}(\sigma_u)>0.
\]

For exact coefficient tests of these four inequalities, put
\[
\lambda=a+\frac{2a^2}{p}+x,
\qquad a=2+\alpha+\beta,
\qquad b=1+\beta,\qquad p=a+q,
\]
where $x,\alpha,\beta,q\ge0$.
Here $\beta=b-1$, $\alpha=a-b-1$, $q=p-a$, and $x=\lambda-a-2a^2/p$ are nonnegative and cover the entire
parameter range. Also $X>0$, $X^2-g'>0$, and $p>0$, so clearing these denominator factors preserves the
signs.
After substitution, the four resulting polynomials in
$(\alpha,\beta,q,x)$ have $692$, $920$, $106$, and $170$
nonzero monomials, respectively; every coefficient is nonnegative.
The last two, which correspond to $\Delta_{C'}$, have positive constant terms. Therefore
$E(\sigma_\ell),E(\sigma_u)\ge0$ and $\Delta_{C'}(\sigma_\ell),\Delta_{C'}(\sigma_u)>0$. By affine
interpolation, these inequalities also hold at $\sigma$, giving
\[
W_{\mathcal R_{g'}}(\lambda)\le C_{g',d}(\lambda).
\]
The four exact endpoint certificates are checked by
\path{python/k1_split_certificate.py} and \path{sage/sage_k1.py};
see Appendix~\ref{app:K1mc_verification}.
\end{proof}

\begin{Lemma}[Transfer of the fractional comparison]\label{lem:transfer}
Let $a\ge4$ and $p\ge a$, and keep the notation introduced above. 
For $\lambda>\rho(\mathcal R_g)$, we have
\[
W_Y(\lambda)\le C_{g,d+1}(\lambda)
\]
whenever $\lambda\ge a+2a^2/p$.
\end{Lemma}

\begin{proof}
The admissibility conditions of (K1) give
\[
k_g+2\le d+1\le g,
\qquad |V(\mathcal R_{g'})|\le d.
\]
Hence the pole bound~\eqref{eq:core_cs_pole} gives
$r_{g,d+1}<\rho(\mathcal R_g)<\lambda$, so
$C_{g,d+1}(\lambda)$ is finite and positive
by~\eqref{eq:core_cs_walk}.

On the other hand, Lemma~\ref{lem:MC} gives
\[
0<W_{\mathcal R_{g'}}(\lambda)
\le C_{g',d}(\lambda)<\infty,
\]
where the finiteness of $C_{g',d}(\lambda)$
follows from $\Delta_{C'}>0$.
By the closure identity~\eqref{eq:core_cs_closure},
\[
\mathscr F_\lambda(d,C_{g',d}(\lambda))
=C_{g,d+1}(\lambda)>0.
\]
Lemma~\ref{lem:transfer_branch} therefore shows
that the Bellman transfer is strictly increasing
between $W_{\mathcal R_{g'}}(\lambda)$ and
$C_{g',d}(\lambda)$. Together with the transfer
identity~\eqref{eq:core_bellman_identity}, this gives
\[
W_Y(\lambda)
=\mathscr F_\lambda(d,W_{\mathcal R_{g'}}(\lambda))
\le\mathscr F_\lambda(d,C_{g',d}(\lambda))
=C_{g,d+1}(\lambda),
\]
as claimed.
\end{proof}

\begin{Lemma}[Strict lower-range polynomial comparison]
\label{lem:MA}
Let $a\ge4$ and $1\le b\le a-1$ be integers,
and let $j\ge1$ and $0\le t\le k-1$ be real,
where $k=a+j$. Retain the relations
\[
d=a+p,\qquad
\binom{k}{2}+t=\binom a2+b+p.
\]
Extend the comparison polynomial
$\mathcal N_{\mathrm A}$ from
\eqref{eq:core_K1_dense_num} to these real
parameters by using the formulas~\eqref{eq:WR}
with $h=k-t$ and $(k',h')=(a-1,a-b)$.
Then
\[
\mathcal N_{\mathrm A}(\lambda)>0
\qquad\text{for }
k-1\le\lambda\le k+\frac{2a^2}{p}.
\]
\end{Lemma}

\begin{proof}
Since $k=a+j$, $j\ge1$, and $t\ge0$, the edge-count relation gives
\[
p=\binom{a+j}{2}+t-\binom a2-b \ge a-b\ge1.
\]
Hence $p>0$, so multiplication by positive powers of $p$ preserves inequalities.
Let
\[
x=\lambda-(k-1),
\qquad X_{\max}=\frac{p+2a^2}{p},
\]
so that $0\le x\le X_{\max}$.

After substituting the edge relation defining $p$ into $\mathcal N_{\mathrm A}$, 
the coefficient of $\lambda^7$ cancels. 
The resulting polynomial has degree at most six in $x$:
\[
\mathcal N_{\mathrm A}(k-1+x)
=\sum_{i=0}^6c_i(a,b,j,t)x^i.
\]
To certify positivity on $[0,X_{\max})$, substitute
\[
x=X_{\max}\frac{y}{1+y},
\qquad y\ge0.
\]
Substitute this expression for $x$ and multiply by the strictly positive factor $p^6(1+y)^6$. 
We obtain the polynomial identity
\[
p^6(1+y)^6\mathcal N_{\mathrm A}
 \left(k-1+X_{\max}\frac{y}{1+y}\right)
=\sum_{i=0}^6
c_i(p+2a^2)^ip^{6-i}y^i(1+y)^{6-i}
=\sum_{\ell=0}^6 B_\ell y^\ell,
\]
where
\[
B_\ell=\sum_{i=0}^\ell
c_i(p+2a^2)^ip^{6-i}
\binom{6-i}{\ell-i}.
\]
Thus, $B_\ell>0$ for all $0\le\ell\le6$ implies $\mathcal N_{\mathrm A}(\lambda)>0$ 
for $\displaystyle k-1\le\lambda<k+\frac{2a^2}{p}$; 
positivity at $\displaystyle \lambda=k+\frac{2a^2}{p}$ will follow from $B_6>0$.

To cover $0\le t<k-1$, let
\[
t=(k-1)\frac{z}{1+z},\qquad z\ge0.
\]
Since
\[
p=aj+\frac{j(j-1)}2+t-b
\]
and the coefficients $c_i$ have degree at most four in $t$, 
each $B_\ell$ has degree at most ten in $t$.
Thus multiplication by $(1+z)^{10}$ clears all denominators introduced by the substitution,
without changing the sign of $B_\ell$.

First suppose that $\delta\ge3$. Let
\[
b=1+\beta,\qquad
a=4+\alpha+\beta,\qquad
j=1+\iota,
\]
where $\alpha,\beta,\iota\ge0$.
These substitutions cover the stated parameter range,
with $\alpha=\delta-3$, $\beta=b-1$ and $\iota=j-1$.
The resulting seven polynomials in $(\alpha,\beta,\iota,z)$ have between $16{,}509$ and $16{,}717$ nonzero monomials, 
all with nonnegative coefficients and strictly positive constant terms.

For the remaining cases $\delta=1$ and $\delta=2$, let
\[
a=4+\alpha,\qquad b=a-\delta,\qquad
j=1+\iota,
\]
where $\alpha,\iota\ge0$.
These two cases yield fourteen further polynomials, each with between $2{,}390$ and $2{,}472$ nonzero monomials, 
again with nonnegative coefficients and strictly positive constant terms.

Altogether, these $21$ coefficient tests prove $B_\ell>0$ for $\ell=0,\ldots,6$ throughout $0\le t<k-1$.

For $t=k-1$, direct substitution into the coefficients $B_\ell$ gives another $21$ polynomials:
seven with $1{,}520$ monomials each for $\delta\ge3$, 
and fourteen with $225$ monomials each for $\delta=1,2$.
Their coefficients are nonnegative and their constant terms are strictly positive.
Hence $B_\ell>0$ also for $t=k-1$.
Thus all seven $B_\ell$ are positive throughout the stated real parameter range. The $21$ interior and
$21$ boundary polynomials are reconstructed from the displayed substitutions and verified coefficient by
coefficient using exact arithmetic; see Appendix~\ref{app:K1ma_verification}.

Because $p^6(1+y)^6>0$, positivity of
$\sum_{\ell=0}^6B_\ell y^\ell$ proves
$\mathcal N_{\mathrm A}>0$ for $0\le x<X_{\max}$.
The endpoint $x=X_{\max}$ corresponds to $y\to\infty$. Its positivity follows from the leading
coefficient:
\[
B_6=p^6\mathcal N_{\mathrm A}(k-1+X_{\max})
   =p^6\mathcal N_{\mathrm A}\left(k+\frac{2a^2}{p}\right)>0.
\]
Since $p>0$, positivity holds also at the right endpoint, completing the proof on the closed interval.

The $21$ interior coefficient tests are implemented in
\path{python/k1_split_certificate.py} and \path{sage/sage_k1.py};
the $21$ boundary tests are implemented in
\path{python/strict_face_MA_sympy.py} and
\path{sage/sage_strict_face_MA.py}.
See Appendix~\ref{app:K1ma_verification} for their verification instructions.
\end{proof}

\begin{proof}[Completion of the proof of part~(c)]
Since $p\ge a$ and $b\ge1$,
\[
g=\binom a2+b+p\ge\binom{a+1}{2}+1.
\]
Hence $k_g=a+j$ for some integer $j\ge1$, so Lemma~\ref{lem:MA} applies with $k=k_g$ and $t=t_g$.

For $\rho(\mathcal R_g)<\lambda\le a+2a^2/p$, 
we have
\[
k_g-1<\lambda\le a+\frac{2a^2}{p}
<k_g+\frac{2a^2}{p},
\]
where the first inequality follows from $\rho(\mathcal R_g)\ge k_g-1$.
Lemma~\ref{lem:MA} therefore gives $\mathcal N_{\mathrm A}(\lambda)>0$, 
and \eqref{eq:core_K1_dense_num} yields
\[
W_Y(\lambda)<W_{\mathcal R_g}(\lambda).
\]

For $\lambda>\rho(\mathcal R_g)$ with $\lambda\ge a+2a^2/p$,
Lemma~\ref{lem:transfer} gives
\[
W_Y(\lambda)\le C_{g,d+1}(\lambda).
\]
The two ranges cover every $\lambda>\rho(\mathcal R_g)$, proving part~(c) and hence (K1).
\end{proof}

\subsection{Strict domination of intermediate cores}

By Theorem~\ref{thm:core_cs_conditional} and the verification of (K2), (J), and (K1), the pointwise
maximum~\eqref{eq:core_pointwise_conjecture} is attained by $\mathcal R_g$ or $\mathcal S_g$. To exclude
ties with any other core, we combine the first-part estimate~\eqref{eq:core_firstpart_strong},
Theorem~\ref{thm:K1}, and Lemma~\ref{lem:strict_prefix}.

\begin{Th}[Strict pointwise inequality]
\label{thm:strict_pointwise}
Let $g\ge4$ and suppose that $H$ is the core of a connected threshold graph with $g$ edges, 
such that $H\not\cong\mathcal R_g,\mathcal S_g$.
Then for every $\lambda>\rho(\mathcal R_g)$,
\[
W_H(\lambda)
<
\max\{W_{\mathcal R_g}(\lambda),W_{\mathcal S_g}(\lambda)\}.
\]
\end{Th}

\begin{proof}
Fix $g\ge4$ and $\lambda>\rho(\mathcal R_g)$,
and let $d$ be the first part of the strict
partition of $H$.
For $d\le k_g$, Theorem~\ref{thm:core_dense_all}
gives $W_H(\lambda)<W_{\mathcal R_g}(\lambda)$,
since $H\not\cong\mathcal R_g$.
For $d=g$, we have $H\cong\mathcal S_g$,
contrary to our assumption.
Thus it remains to consider
\[
k_g+1\le d\le g-1,
\]
which is possible only for $g\ge5$.

Let $g'=g-d$. 
Since the inner core of $H$ has at most $d$ vertices and 
$\mathcal R_{g'}$ has minimum order among graphs with $g'$ edges, 
we have $|V(\mathcal R_{g'})|\le d$.
The first-part estimate~\eqref{eq:core_firstpart_strong} and the transfer identity~\eqref{eq:core_K1_Y_transfer}
therefore give
\begin{equation}\label{eq:strict_firstpart}
W_H(\lambda)
\le\max\{
\mathscr F_\lambda(d,W_{\mathcal R_{g'}}(\lambda)),
C_{g,d+1}(\lambda)\}
=\max\{W_Y(\lambda),C_{g,d+1}(\lambda)\}.
\end{equation}
By Theorem~\ref{thm:K1}, either
\[
W_Y(\lambda)<W_{\mathcal R_g}(\lambda)
\quad\text{or}\quad
W_Y(\lambda)\le C_{g,d+1}(\lambda).
\]
Since $k_g+2\le d+1\le g$, Lemma~\ref{lem:strict_prefix} gives
\[
C_{g,d+1}(\lambda)
<\max\{W_{\mathcal R_g}(\lambda),
        C_{g,g+1}(\lambda)\}
=\max\{W_{\mathcal R_g}(\lambda),
        W_{\mathcal S_g}(\lambda)\}.
\]
In either case we obtain the same strict bound 
$W_H(\lambda)<\max\{W_{\mathcal R_g}(\lambda),W_{\mathcal S_g}(\lambda)\}$.
\end{proof}

\section{Proof of Theorem~\ref{thm:main}}
\label{subsec:core_main_complete}

Fix a feasible pair $(n,m)$ and put $e=m-n+1$.
By Brualdi and Solheid~\cite{brusol1986}, every extremal connected graph is threshold. 
Thus it suffices to consider graphs
\[
G_H=K_1\vee[H\cup(n-1-|V(H)|)K_1],
\]
where $H$ is a threshold core with $e$ edges and $|V(H)|\le n-1$.
For $e\ge1$ and $L>\rho(\mathcal R_e)$, Proposition~\ref{prop:core_comparison} gives
\[
W_H(L)<L^2-(n-1)
\quad\Longrightarrow\quad
\rho(G_H)<L.
\]

For $e=0$, the only connected threshold graph with $n-1$ edges is $K_{1,n-1}=D_{n,0}$.
For $e=1,2$, Lemma~\ref{lem:core_strict_partitions} shows that the unique core is $\mathcal R_e=\mathcal S_e$, 
so $D_{n,e}$ is again the unique extremal graph.
For $e=3$, the two strict partitions $(3)$ and $(2,1)$ 
correspond to $\mathcal S_3=K_{1,3}$ and $\mathcal R_3=K_3$, respectively.
Since
\[
W_{K_3}(\lambda)-W_{K_{1,3}}(\lambda)
=\frac{6}{(\lambda-2)(\lambda^2-3)}>0
\qquad(\lambda>2),
\]
we have, for $L=\rho(D_{n,3})>2$,
\[
W_{K_{1,3}}(L)<W_{K_3}(L)=L^2-(n-1).
\]
Thus the star core, whenever admissible, gives a smaller spectral radius than $D_{n,3}$.
Hence $D_{n,e}$ is the unique extremal graph for $e\le3$.

Suppose now that $e\ge4$.

\emph{First suppose that $n\le e+1$.}
The star core $\mathcal S_e$ is then inadmissible, while $\mathcal R_e$ is admissible by feasibility.
Let $L=\rho(D_{n,e})>\rho(\mathcal R_e)$.
Theorem~\ref{thm:vstar} and Lemma~\ref{lem:core_virtual_star},
together with the cone identity \eqref{eq:core_cone_equation}, 
give
\[
W_{\mathcal S_e}(L)<W_{\mathcal R_e}(L)=L^2-(n-1).
\]
For every other admissible core $H$, Theorem~\ref{thm:strict_pointwise} gives
\[
W_H(L)
<\max\{W_{\mathcal R_e}(L),
         W_{\mathcal S_e}(L)\}
=W_{\mathcal R_e}(L)=L^2-(n-1).
\]
Proposition~\ref{prop:core_comparison} now gives $\rho(G_H)<\rho(D_{n,e})$, 
so $D_{n,e}$ is the unique extremal graph.

\emph{Now suppose that $n\ge e+2$.}
Both $\mathcal R_e$ and $\mathcal S_e$ are admissible.
Let
\[
L=\max\{\rho(D_{n,e}),\rho(V_{n,e})\}
>\rho(\mathcal R_e).
\]
By Proposition~\ref{prop:core_comparison} and
the cone identity~\eqref{eq:core_cone_equation},
\[
\max\{W_{\mathcal R_e}(L),
       W_{\mathcal S_e}(L)\}=L^2-(n-1).
\]
For every other admissible core $H$,
Theorem~\ref{thm:strict_pointwise} gives
\[
W_H(L)
<\max\{W_{\mathcal R_e}(L),
        W_{\mathcal S_e}(L)\}=L^2-(n-1),
\]
and hence $\rho(G_H)<L$ by
Proposition~\ref{prop:core_comparison}.
Thus only $D_{n,e}$ and $V_{n,e}$ can be extremal.

Damnjanovi\'c's Propositions~4.1--4.2
~\cite{damnjanovic2025filomat}
show that $D_{n,e}$ has larger spectral radius
for $n<\omega_e$, $V_{n,e}$ for $n>\omega_e$,
and that they have equal spectral radii precisely
when $n=\omega_e$.
Since $\omega_e>e+2$, the result for $n\le e+1$
also lies in the range $n<\omega_e$,
completing the proof.\hfill\qedsymbol

\appendix
\section{Exact certificates and reproducibility}
\label{app:section6_runners}

The proof uses exact algebraic certificates for rational identities,
polynomial coefficient signs, and Sturm root exclusions. The corresponding
Python/SymPy and SageMath programs reconstruct the same certificate families
using different exact-arithmetic implementations. The graph-theoretic
reductions and the deductions from these certificates are proved in the text.

The scripts are listed below. The time classifications refer to measured
Python executions on the system used to prepare the v37 archive:
short (under 15 seconds), medium (15--60 seconds), and long (over 60 seconds).
The SageMath runtimes may differ. The output strings identify successful
components, but a complete verification also requires a zero exit status
and the final marker of the corresponding runner.

\begingroup\footnotesize
\setlength{\tabcolsep}{2.5pt}
\renewcommand{\arraystretch}{1.28}
\begin{longtable}{@{}>{\raggedright\arraybackslash}p{0.24\linewidth}>{\raggedright\arraybackslash}p{0.33\linewidth}>{\raggedright\arraybackslash}p{0.33\linewidth}r@{}}
\hline
Claim & Exact scripts (Python / SageMath) & Expected output substring & Time \\
\hline
\endfirsthead
\hline Claim & Exact scripts (Python / SageMath) & Expected output substring & Time \\
\hline
\endhead
Virtual-star inequality &
\path{vstar_certificate.py}\newline\path{sage_vstar.py} &
\texttt{PASS (V*)};\newline\texttt{PASS V* finite Sturm, 1766 cases} & med. \\
(K2), $N\le6$ (five cases) &
\path{k2_proof_checks.py}\newline\path{sage_J_K2.py K2} &
\texttt{PASS (K2)} & short \\
(J), 38 small cases and ten Sturm norms &
\path{j_proof_checks.py}\newline\path{sage_J_K2.py J} &
\texttt{PASS (J)} & short \\
(K1)(a), 431/470 monomials &
\path{k1_mregion_certificate.py}\newline\path{sage_k1.py a} &
\texttt{PASS K1(a)} & short \\
(K1)(MC), four endpoint checks &
\path{k1_split_certificate.py}\newline\path{sage_k1.py mc} &
\texttt{PASS K1(MC/MA)} (Python);\newline\texttt{PASS K1(MC)} (SageMath) & med. \\
(K1)(MA), 21 interior polynomials &
\path{k1_split_certificate.py}\newline\path{sage_k1.py ma} &
\texttt{PASS K1(MC/MA)} (Python);\newline\texttt{PASS K1(MA), all 21} (SageMath) & long \\
(K1)(MA), 21 boundary polynomials &
\path{strict_face_MA.py}\newline\path{sage_strict_face_MA.py} &
\texttt{PASS strict (MA) on the face} & short \\
Analytic identities and quotient matrices &
\path{analytic_identities_check.py},\newline\path{templates_check.py} /\newline\path{sage_analytic_identities.py},\newline\path{sage_quotients.py} &
\texttt{48 of 48 identity checks passed};\newline\texttt{ALL symbolic quotient checks PASS} & med. \\
Cubic and exact tie examples &
\path{omega_threshold_identity_check.py},\newline\path{v37_exact_ties.py} /\newline\path{sage_v37_exact_ties.py} &
Exact assertions and\newline\texttt{PASS v37 exact three tie examples} & short \\
\hline
\end{longtable}
\endgroup

\paragraph{Virtual-star certificate.}\label{app:vstar_verification}
For $k\ge60$ the four transformed numerator polynomials have $60$ nonzero
monomials each, all with nonnegative coefficients and positive constant
terms. The highest $y$-coefficient in the fourth polynomial covers the
endpoint $t=k-1$. For $4\le e\le1769$ the programs obtain rational bounds
for the relevant spectral interval and prove, by exact Sturm root counts,
that the difference numerator has no root in that interval and is positive
at a rational endpoint. All $1{,}766$ edge counts are checked.

\paragraph{(K2) and (J).}
\label{app:K2_verification}\label{app:J_verification}
The (K2) programs verify the rational identity~\eqref{eq:K2id},
the polynomial bounds used in the analytic argument for $N\ge7$,
and the four strict cases with $N\le6$ by exact Sturm root counts and quadratic-field sign checks.
The remaining case $(g,d)=(9,5)$ gives equality.

The (J) programs verify positive-coefficient polynomials for the unbounded parameter range, 
exact signs for the $38$ remaining integer pairs, and 
root-free derivative-norm comparisons for ten exceptional edge counts by Sturm's theorem.
The Python (J) program also checks the derivative denominator.

\paragraph{(K1) certificate families.}
\label{app:K1a_verification}\label{app:K1mc_verification}
\label{app:K1ma_verification}
For (K1)(a), the two substitutions yield polynomials with $431$ and $470$ nonzero monomials, respectively,
all with positive integer coefficients.
For (MC), four endpoint polynomials have $692$, $920$, $106$ and $170$ nonzero monomials;
the last two have strictly positive constant terms and establish positivity of the fractional denominator.

For (MA), the coefficients $B_0,\ldots,B_6$ of Lemma~\ref{lem:MA} yield $21$ polynomial certificates for $0\le t<k-1$, 
with $16{,}509$--$16{,}717$ nonzero monomials for $\delta\ge3$ and $2{,}390$--$2{,}472$ for $\delta=1,2$.
The endpoint $t=k-1$ requires another $21$ polynomials, 
with $1{,}520$ nonzero monomials for $\delta\ge3$ and $225$ for $\delta=1,2$.
All have nonnegative coefficients and strictly positive constant terms.
The Python runner uses the exact-rational implementation \path{strict_face_MA.py} for the endpoint;
\path{strict_face_MA_sympy.py} provides an independent SymPy verification.

\paragraph{Reproducing the computations.}
The manuscript version is v37 and is distributed separately from the \path{anc/} directory. 
The verification archive and source files are available at \cite{zenodo}.
The Python environment used for the enclosed logs was Python 3.13.5, SymPy~1.14.0 and mpmath~1.3.0;
\path{anc/python/requirements.txt} gives the required packages.
The complete Python run took 638 seconds (10.6 minutes) on the machine used for verification;
individual times appear in \path{anc/documentation/PYTHON_TIMINGS_V37.md}.

\begin{verbatim}
# From the anc/ directory:
sha256sum -c SHA256SUMS.txt
# On macOS: shasum -a 256 -c SHA256SUMS.txt

# Complete exact Python suite (from anc/python/):
cd python
sh run_proof_only.sh > ../documentation/logs/python_v37_full.log 2>&1
# Final marker: PASS: completed all v37 Python exact certificate checks

# Complete exact SageMath suite (from anc/sage/):
cd ../sage
sh run_sage_checks.sh > ../documentation/logs/sage_v37_full.log 2>&1
# Final marker: PASS: completed all v37 SageMath exact certificate checks
\end{verbatim}

The archive contains logs of complete Python and SageMath verification runs. 
The Python run terminated with zero exit status, 
while the independently executed SageMath run successfully completed all prescribed checks.
Both implementations use exact arithmetic, in different computer algebra systems, 
to verify the same mathematical certificates. 
These computations do not replace the analytic proofs of compression, induction, strict endpoint domination, 
or Damnjanovi\'c's comparison of the two candidates.

\paragraph{Use of generative AI tools.}
During the development and preparation of this work,
the author used OpenAI's ChatGPT and Anthropic's Claude
for mathematical exploration, discussion and refinement of proof strategies, checking algebraic manipulations,
assistance in developing and debugging Python/SymPy and SageMath verification programs, and revision of the exposition.
In particular, AI-assisted discussions concerned the core-resolvent and Bellman-recursion approach in Section 2,
the fractional complete-split bounds in Section 5,
the verification of the inequalities (J), (K1), and (K2)
and the treatment of equality cases in Section 6.
The author has independently examined the mathematical arguments and computational certificates and
takes full responsibility for the correctness and originality of the results,
the accuracy of the references, and the final manuscript.

\end{document}